\documentclass[11pt]{article}

\usepackage[margin=1in]{geometry}
\usepackage{amsmath,amssymb,amsthm}
\usepackage{booktabs,array,longtable}
\usepackage{enumitem}
\usepackage{graphicx}
\usepackage{tikz}
\usepackage{placeins}
\usepackage{cite}
\usepackage[hidelinks,hyperfootnotes=false]{hyperref}
\usepackage{microtype}
\allowdisplaybreaks

\setlist[enumerate]{topsep=3pt,itemsep=1pt,parsep=0pt,partopsep=0pt}

\theoremstyle{plain}
\newtheorem{theorem}{Theorem}[section]
\newtheorem{lemma}[theorem]{Lemma}
\newtheorem{proposition}[theorem]{Proposition}
\newtheorem{corollary}[theorem]{Corollary}
\theoremstyle{definition}

\theoremstyle{remark}

\DeclareMathOperator{\Min}{Min}
\DeclareMathOperator{\Max}{Max}
\newcommand{\Crit}{\operatorname{Crit}}
\newcommand{\Inc}{\operatorname{Inc}}
\newcommand{\Vee}{\mathsf V}
\newcommand{\strictdown}{\mathord{\downarrow}^{\!\circ}}
\newcommand{\strictup}{\mathord{\uparrow}^{\!\circ}}
\newcommand{\fundingcontact}{%
\parbox[t]{0.94\linewidth}{%
This work is supported by the National Natural Science Foundation of China
(Grant nos.~12471438 and 12331016) and the Fundamental Research Funds for the
Central Universities (Grant no.~GK202501014).\\
\textsuperscript{*}Corresponding author.\\
Email addresses:
\href{mailto:dzc0419@gmail.com}{dzc0419@gmail.com} (Zhaochen Dong),
\href{mailto:wangkaiyun@snnu.edu.cn}{wangkaiyun@snnu.edu.cn}
(Kaiyun Wang).}}

\tikzset{
  every path/.style={line width=1.3pt},
  pn/.style={circle,draw,fill=white,inner sep=1.5pt,line width=1.3pt},
  every label/.append style={font=\small},
  title/.style={draw=none,fill=none,font=\large}
}

\title{The Kelly--Trotter product conjecture for posets of dimension three}
\author{%
Zhaochen Dong,\quad Kaiyun Wang\thanks{\fundingcontact}\\[0.7ex]
{\small\itshape School of Mathematics and Statistics, Shaanxi Normal
University,}\\[-0.1ex]
{\small\itshape Xi'an 710119, Shaanxi, P. R. China}}
\date{}

\begin{document}
\maketitle

\begin{abstract}
Kelly and Trotter conjectured that
\(\dim(P\times Q)\geq \dim P+\dim Q-2\) for all finite posets \(P\) and
\(Q\). We prove the conjecture when \(\dim P=\dim Q=3\). This also
disproves Trotter's conjecture that, for every \(1\leq m\leq n\), there
exist finite posets \(P\) and \(Q\) with
\(\dim P=m\), \(\dim Q=n\), and \(\dim(P\times Q)=n\).
We further prove that \(\dim(C_k\times P)=4\) for every finite poset
\(P\) with \(\dim P=3\) and every crown \(C_k\) with \(k\geq3\).
The proof uses the classification of \(3\)-irreducible posets and graphs
of critical pairs. For the six infinite noncrown families, we construct
explicit non-\(3\)-colorable subgraphs. The ten fixed posets are handled
by an exhaustive \(3\)-coloring search.
\end{abstract}

\medskip
\noindent\textbf{Keywords.}
order dimension;  \(3\)-irreducible posets;
crowns; graphs of critical pairs.

\smallskip
\noindent\textbf{2020 Mathematics Subject Classification.}
Primary 06A07; Secondary 05C15.

\section{Introduction}\label{sec:introduction}

For finite posets \(P\) and \(Q\),
\[
  \max\{\dim P,\dim Q\}
  \leq \dim(P\times Q)
  \leq \dim P+\dim Q.
\]
Kelly and Trotter conjectured the stronger lower bound
\begin{equation}\label{eq:product-conjecture}
  \dim(P\times Q)\geq \dim P+\dim Q-2
\end{equation}
for all finite posets \(P\) and \(Q\)~\cite{KellyTrotter1982}.
No counterexample is known~\cite{Bergman2026}.

In~\cite{Trotter1985}, Trotter proposed a competing conjecture: for
every \(1\leq m\leq n\), there exist finite posets \(P\) and \(Q\)
such that
\[
  \dim P=m,\qquad
  \dim Q=n,\qquad
  \dim(P\times Q)=n.
\]
For \(m=n=3\), this would require a product of dimension three,
whereas \eqref{eq:product-conjecture} requires every such product to
have dimension at least four.

For a positive integer \(m\), a poset \(P\) is \(m\)-irreducible if
\(\dim P=m\) and \(\dim(P-x)<m\) for every \(x\in P\).
Every finite poset \(P\) with \(\dim P=m\) contains an
\(m\)-irreducible subposet~\cite{Kelly1977}.  Thus, by monotonicity,
for factors of dimension three we may pass to \(3\)-irreducible
subposets.  Kelly and, independently, Trotter and Moore classified the
\(3\)-irreducible posets~\cite{Kelly1977,TrotterMoore1976}.  Up to
isomorphism and duality, the list consists of crowns, six infinite
noncrown families, and ten fixed posets.

Let \(\Vee\) be the poset on \(\{0,p,q\}\) with \(0<p\) and \(0<q\),
and let \(\Vee^d\) be its dual.  Let \(Z_3\), \(Z_3^d\), and \(C_2\)
be the remaining posets in Figure~\ref{fig:small-posets}.
All five posets in the figure have dimension two, and \(C_2\) is
self-dual.

For \(k\geq3\), let \(C_k\) denote the crown, the height-two poset whose
Hasse diagram is a cycle of length \(2k\).  In
Figure~\ref{fig:kelly-An}, these posets are labeled according to Kelly's
notation \(A_n=C_{n+3}\).  We have
\[
  \dim C_k=3 \quad\text{for } k\geq3,
\]
and every \(C_k\) contains copies of both \(\Vee\) and \(\Vee^d\)
\cite{BarreraCruzEtAl2019,TrotterCrown1974}.

\begin{figure}[htbp]
\centering
\begin{tikzpicture}[x=1cm,y=1cm]
  \begin{scope}[xshift=0cm]
    \node[pn,label=below:$0$] (v0) at (0,0) {};
    \node[pn,label=above left:$p$] (vp) at (-0.65,1.2) {};
    \node[pn,label=above right:$q$] (vq) at (0.65,1.2) {};
    \draw (v0)--(vp) (v0)--(vq);
    \node at (0,-0.75) {$\Vee$};
  \end{scope}
  \begin{scope}[xshift=3cm]
    \node[pn,label=below left:$p$] (vdp) at (-0.65,0) {};
    \node[pn,label=below right:$q$] (vdq) at (0.65,0) {};
    \node[pn,label=above:$0$] (vd0) at (0,1.2) {};
    \draw (vdp)--(vd0) (vdq)--(vd0);
    \node at (0,-0.75) {$\Vee^d$};
  \end{scope}
  \begin{scope}[xshift=6.3cm]
    \node[pn] (z0) at (0,0) {};
    \node[pn] (za) at (-0.9,1.2) {};
    \node[pn] (zb) at (0,1.2) {};
    \node[pn] (zc) at (0.9,1.2) {};
    \draw (z0)--(za) (z0)--(zb) (z0)--(zc);
    \node at (0,-0.75) {$Z_3$};
  \end{scope}
  \begin{scope}[xshift=10.1cm]
    \node[pn] (zda) at (-0.9,0) {};
    \node[pn] (zdb) at (0,0) {};
    \node[pn] (zdc) at (0.9,0) {};
    \node[pn] (zd0) at (0,1.2) {};
    \draw (zda)--(zd0) (zdb)--(zd0) (zdc)--(zd0);
    \node at (0,-0.75) {$Z_3^d$};
  \end{scope}
  \begin{scope}[xshift=13.7cm]
    \node[pn] (ca1) at (-0.65,0) {};
    \node[pn] (ca2) at (0.65,0) {};
    \node[pn] (cb1) at (-0.65,1.2) {};
    \node[pn] (cb2) at (0.65,1.2) {};
    \draw (ca1)--(cb1) (ca1)--(cb2) (ca2)--(cb1) (ca2)--(cb2);
    \node at (0,-0.75) {$C_2$};
  \end{scope}
\end{tikzpicture}
\caption{The posets \(\Vee\), \(\Vee^d\), \(Z_3\), \(Z_3^d\), and \(C_2\).}
\label{fig:small-posets}
\end{figure}
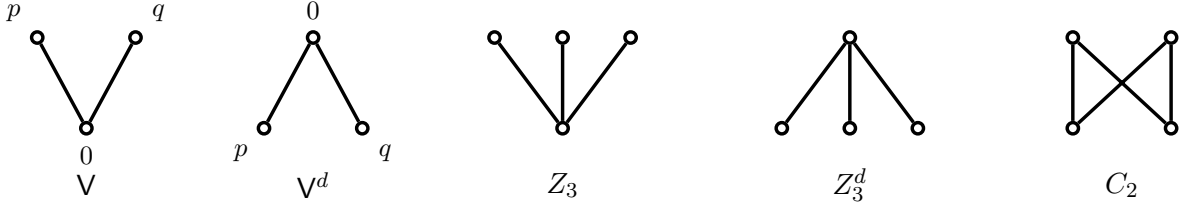

In~\cite{Reuter1989}, Reuter proved \(\dim(P\times P)\geq4\) whenever
\(\dim P=3\).  Using the classification of \(3\)-irreducible posets,
Felsner et al.\ later gave another proof of Reuter's
result~\cite{OrderGridsProducts}.  They also proved that
\[
  \dim(C_k\times C_\ell)=4 \quad\text{for } k,\ell\geq3.
\]

After passing to \(3\)-irreducible subposets, the problem therefore
naturally divides into three cases: crown--crown, crown--noncrown, and
noncrown--noncrown.  The crown--crown case is settled by the result
above.  For the noncrown--noncrown case, Felsner et al.\ observed that
every noncrown member of the classification contains a copy of
\(Z_3\), \(Z_3^d\), or \(C_2\), and proved the required lower bound for
\[
  Z_3\times Z_3,\qquad
  Z_3\times Z_3^d,\qquad
  Z_3\times C_2,\qquad
  C_2\times C_2.
\]
By duality, symmetry, and monotonicity, this settles the
noncrown--noncrown case.  Hence only the mixed crown--noncrown case
remains.

This remaining case cannot be treated by the same reduction.  Up to
duality, such a reduction would require
\[
  \dim(C_k\times Z_3)\geq4
  \qquad\text{and}\qquad
  \dim(C_k\times C_2)\geq4,
\]
but both inequalities are false.
Indeed, they proved that for every poset \(X\) and every \(k\geq3\),
\begin{equation}\label{eq:fmw-crown-upper-intro}
  \dim(X\times C_k)
  \leq\max\bigl\{3,\dim X+1\bigr\}.
\end{equation}
If \(R\) is a noncrown \(3\)-irreducible poset and \(X\) is a nonempty
proper subposet of \(R\), then \(\dim X\leq2\).  The general lower
bound together with \eqref{eq:fmw-crown-upper-intro} gives
\[
  \dim(C_k\times X)=3.
\]
Consequently, no proper subposet of \(R\) gives the required lower
bound through monotonicity.  Thus the mixed case must retain the full
noncrown factor \(R\).
Nevertheless, their computer calculations gave dimension four for
every tested crown--noncrown product of \(3\)-irreducible posets.
For the six infinite families, the cases \(n=0,1,2\) were computed
(see \cite[Table~1]{OrderGridsProducts}).

We prove the Kelly--Trotter conjecture when both factors have dimension
three.

\medskip
\noindent\textbf{Theorem~\ref{thm:main}.}\quad
\textit{Let \(P\) and \(Q\) be finite posets.  If
\(\dim P=\dim Q=3\), then}
\[
  \dim(P\times Q)\geq4.
\]

This result disproves Trotter's conjecture.  Indeed, the case
\(m=n=3\) would require \(\dim(P\times Q)=3\).

Following Felsner and Trotter~\cite{FelsnerTrotter2000}, we use the
graph of critical pairs to obtain lower bounds on dimension.  For a
poset \(Y\), let \(\Gamma(Y)\) denote its graph of critical pairs.  If
\(\dim Y=3\), then \(\Gamma(Y)\) is \(3\)-colorable.  Thus a
non-\(3\)-colorable subgraph of \(\Gamma(Y)\) implies
\(\dim Y\geq4\).

To simplify the proof of the mixed case, we work with
\(\Vee\times R\) or \(\Vee^d\times R\).  For every noncrown
\(3\)-irreducible poset \(R\),
we prove that at least one of
\(\Gamma(\Vee\times R)\) and \(\Gamma(\Vee^d\times R)\) is not
\(3\)-colorable.  By monotonicity, this gives
\(\dim(C_k\times R)\geq4\) for every \(k\geq3\).
For each of the six infinite noncrown families, we construct an
explicit non-\(3\)-colorable subgraph of \(\Gamma(\Vee\times R)\)
for every \(n\geq0\).  For each of the ten fixed posets, we verify by
an exhaustive \(3\)-coloring search that \(\Gamma(\Vee\times R)\) is
not \(3\)-colorable.  By duality, these results also cover the
corresponding dual posets.
Together with the crown--crown and noncrown--noncrown cases, the mixed
case proves Theorem~\ref{thm:main}.

\section{Preliminaries}\label{sec:preliminaries}

Throughout, all posets are finite.  We write \(x\parallel_P y\) when
\(x\) and \(y\) are \emph{incomparable} in \(P\).  A
\emph{linear extension} of \(P\) is a linear order \(L\) on the ground
set of \(P\) such that \(x\leq_Py\) implies \(x\leq_Ly\).  A nonempty
family \(\mathcal R\) of linear extensions is a \emph{realizer} of
\(P\) if
\[
  x\leq_Py
  \quad\Longleftrightarrow\quad
  x\leq_Ly
  \quad\text{for every }L\in\mathcal R.
\]
The \emph{dimension} of \(P\), denoted by \(\dim P\), is the least size
of a realizer of \(P\)~\cite{DushnikMiller1941,Trotter1992}.

The \emph{Cartesian product} \(P\times Q\) carries the coordinatewise
order.  The \emph{dual} \(P^d\) is defined by
\(x\leq_{P^d}y\) if and only if \(y\leq_Px\).  We use the standard
facts that dimension is monotone under taking subposets, that
\(\dim P^d=\dim P\), and that \(P\times Q\cong Q\times P\).  We also
have $  (P\times Q)^d=P^d\times Q^d$.

Let $  \Inc(P)=\bigl\{(x,y)\in P\times P:x\parallel_Py\bigr\}$.
A linear extension \(L\) \emph{reverses} \((x,y)\in\Inc(P)\) if
\(y<_Lx\).  For \(x\in P\), write
\[
\begin{aligned}
  \downarrow x
    &=\bigl\{u\in P:u\leq_Px\bigr\},
  &
  \strictdown x
    &=\bigl\{u\in P:u<_Px\bigr\},\\
  \uparrow x
    &=\bigl\{u\in P:x\leq_Pu\bigr\},
  &
  \strictup x
    &=\bigl\{u\in P:x<_Pu\bigr\}.
\end{aligned}
\]
The sets of \emph{minimal} and \emph{maximal} elements of \(P\) are
denoted by
\[
  \Min(P)=\bigl\{x\in P:\strictdown x=\varnothing\bigr\},
  \qquad
  \Max(P)=\bigl\{x\in P:\strictup x=\varnothing\bigr\}.
\]

A pair \((x,y)\in\Inc(P)\) is \emph{critical} if
\begin{equation}\label{eq:critical-pair}
  \strictdown x\subseteq\strictdown y,
  \qquad
  \strictup y\subseteq\strictup x.
\end{equation}
The set of critical pairs is denoted by \(\Crit(P)\).  A family of
linear extensions realizes \(P\) if and only if every critical pair is
reversed by some member of the family~\cite{KellyTrotter1982}.

Following Felsner and Trotter~\cite{FelsnerTrotter2000}, the
\emph{graph of critical pairs} \(\Gamma(P)\) has vertex set
\(\Crit(P)\).  Distinct vertices \((x_1,y_1)\) and \((x_2,y_2)\) are
\emph{adjacent} if and only if
\begin{equation}\label{eq:critical-edge}
  x_1\leq_Py_2
  \quad\text{and}\quad
  x_2\leq_Py_1.
\end{equation}
We write \(\alpha\sim\beta\) when \(\alpha\) and \(\beta\) are adjacent;
equivalently, \(\{\alpha,\beta\}\) is an edge of \(\Gamma(P)\).  In this
case, \(\alpha\) and \(\beta\) are neighbors.

For a positive integer \(k\), a \emph{proper \(k\)-coloring} of a graph
assigns one of \(k\) colors to each vertex so that adjacent vertices
receive different colors.  A graph is \emph{\(k\)-colorable} if it has
a proper \(k\)-coloring.

The following result is due to Felsner and Trotter
\cite[Lemma~3.3]{FelsnerTrotter2000}.

\begin{lemma}\label{lem:coloring-lower-bound}
If \(\dim P=d\), then \(\Gamma(P)\) is \(d\)-colorable.
\end{lemma}

\begin{corollary}\label{cor:nonthreecolorable-lower-bound}
If a subgraph of \(\Gamma(P)\) is not \(3\)-colorable, then
\(\dim P\geq4\).
\end{corollary}

\begin{proof}
Suppose that \(\dim P\leq3\).  By
Lemma~\ref{lem:coloring-lower-bound}, \(\Gamma(P)\) has a proper
\(3\)-coloring.  Its restriction to any subgraph is also a proper
\(3\)-coloring, a contradiction.
\end{proof}

\begin{lemma}\label{lem:graph-duality}
For every poset \(P\), the map
\[
  \Gamma(P)\longrightarrow\Gamma(P^d),
  \qquad
  (x,y)\longmapsto(y,x),
\]
is a graph isomorphism.
\end{lemma}

\begin{proof}
By \eqref{eq:critical-pair}, \((x,y)\) is critical in \(P\) if and
only if \((y,x)\) is critical in \(P^d\).  Therefore, the map is  a
bijection on critical pairs.  By \eqref{eq:critical-edge}, two critical
pairs are adjacent if and only if their images are adjacent.
\end{proof}

We use Kelly's notation for the \(3\)-irreducible posets.  The crowns
are \(A_n=C_{n+3}\), the six infinite noncrown families are
\(E_n,F_n,G_n,H_n,I_n,J_n\), and the ten fixed posets are
\[
  B,C,D,CX_1,CX_2,CX_3,EX_1,EX_2,FX_1,FX_2.
\]
The diagrams below use the labels from Kelly's classification
\cite{Kelly1977}.

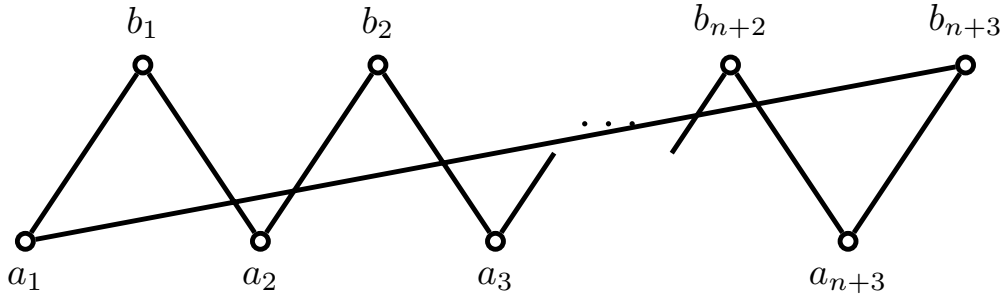
\begin{figure}[htbp]
\centering
\resizebox{0.82\textwidth}{!}{%
\begin{tikzpicture}[scale=1.1]
\node[pn, label=below:$a_1$] (a1) at (0,0) {};
\node[pn, label=below:$a_2$] (a2) at (2,0) {};
\node[pn, label=below:$a_3$] (a3) at (4,0) {};
\node[pn, label=below:$a_{n+3}$] (an3) at (7,0) {};
\node[pn, label=above:$b_1$] (b1) at (1,1.5) {};
\node[pn, label=above:$b_2$] (b2) at (3,1.5) {};
\node[pn, label=above:$b_{n+2}$] (bn2) at (6,1.5) {};
\node[pn, label=above:$b_{n+3}$] (bn3) at (8,1.5) {};

\draw
  (a1)--(b1)
  (a1)--(bn3)
  (a2)--(b1)
  (a2)--(b2)
  (a3)--(b2);
\draw
  (an3)--(bn2)
  (an3)--(bn3);
\draw (a3)--(4.5,0.75);
\draw (5.5,0.75)--(bn2);
\node at (5,1) {\Large $\dots$};
\end{tikzpicture}%
}
\caption{The crown \(A_n=C_{n+3}\).}
\label{fig:kelly-An}
\end{figure}

\section{Critical pairs in \texorpdfstring{\(\Vee\times R\)}{V x R}}
\label{sec:three-point}

The following lemma gives the two cases used below.

\begin{lemma}\label{lem:product-critical}
Let \(R\) be a poset and let \(x,y\in R\).
\begin{enumerate}[label=\textup{(\roman*)},leftmargin=*]
\item
For each \(s\in\{p,q\}\), the pair
\(\bigl((0,x),(s,y)\bigr)\) is critical in \(\Vee\times R\) if and
only if
\begin{equation}\label{eq:product-critical-first}
  x\nleq_Ry,
  \quad
  \strictdown x\subseteq\downarrow y,
  \quad
  \strictup y\subseteq\uparrow x.
\end{equation}

\item
For distinct \(s,t\in\{p,q\}\), the pair
\(\bigl((s,x),(t,y)\bigr)\) is critical in \(\Vee\times R\) if and
only if
\begin{equation}\label{eq:product-critical-second}
  x\leq_Ry,
  \quad
  x\in\Min(R),
  \quad
  y\in\Max(R).
\end{equation}
\end{enumerate}
\end{lemma}

\begin{proof}
For \textup{(i)}, the two product points are
incomparable if and only if \(x\nleq_Ry\).  The conditions in
\eqref{eq:critical-pair} give the two remaining inclusions in
\eqref{eq:product-critical-first}.
For \textup{(ii)}, incomparability is automatic because
\(p\parallel_{\Vee}q\).  The remaining conditions in
\eqref{eq:critical-pair} are exactly
\eqref{eq:product-critical-second}.
\end{proof}

If \(x\in\Min(R)\) and \(y\in\Max(R)\), then
\eqref{eq:product-critical-first} reduces to \(x\nleq_Ry\).
Let
\[
  \alpha=\bigl((s,x),(t,y)\bigr),
  \qquad
  \beta=\bigl((s',x'),(t',y')\bigr)
\]
be distinct critical pairs in \(\Vee\times R\).  Since the product
order is coordinatewise, by \eqref{eq:critical-edge}, we have 
\begin{equation}\label{eq:product-edge}
  \alpha\sim\beta
  \quad\Longleftrightarrow\quad
  s\leq_{\Vee}t',\quad
  s'\leq_{\Vee}t,\quad
  x\leq_Ry',\quad
  x'\leq_Ry.
\end{equation}
In the applications below, the first two comparisons in
\eqref{eq:product-edge} follow immediately from \(0<p\), \(0<q\), or
equality.  Thus only the last two comparisons need to be checked
explicitly.

\begin{proposition}\label{prop:crown-product}
Let \(R\) be a poset with \(\dim R=3\).  If at least one of
\(\Gamma(\Vee\times R)\) and \(\Gamma(\Vee^d\times R)\) is not
\(3\)-colorable, then
\(\dim(C_k\times R)=4\) for every \(k\geq3\).
\end{proposition}
\begin{proof}
By Corollary~\ref{cor:nonthreecolorable-lower-bound}, we have \(\dim(\Vee\times R)\geq4\) or
\(\dim(\Vee^d\times R)\geq4\).  Since every \(C_k\) contains copies of
both \(\Vee\) and \(\Vee^d\), monotonicity gives
\(\dim(C_k\times R)\geq4\).  The reverse inequality follows from
\eqref{eq:fmw-crown-upper-intro} with \(X=R\).
\end{proof}

\section{The six infinite families}\label{sec:infinite}

The classification contains six infinite families
\(E_n,F_n,G_n,H_n,I_n,J_n\).  For each member \(R\) of these families,
the corresponding subsection gives a labeled Hasse diagram of \(R\)
and two tables defining a graph \(W_R\).  The left table lists the
vertices of \(W_R\) as ordered pairs in \(\Vee\times R\).  In the
right table, every vertex in the first column is adjacent to every
vertex in the second column, with the stated range applying to the
entire row.  Each adjacency is recorded only once.  We write
\(x\sim y_1,\ldots,y_k\) to mean that \(x\) is adjacent to each of
\(y_1,\ldots,y_k\).

\begin{theorem}\label{thm:infinite}
For every \(n\geq0\) and every \(R\in\{E_n,F_n,G_n,H_n,I_n,J_n\}\), the graph \(\Gamma(\Vee\times R)\) is not \(3\)-colorable.
\end{theorem}

\subsection{The family \texorpdfstring{\(E_n\)}{E-n}}

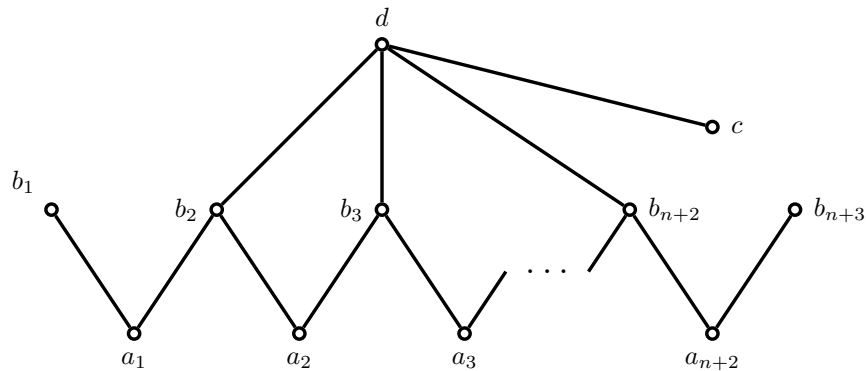
\begin{figure}[htbp]
\centering
\resizebox{0.70\textwidth}{!}{%
\begin{tikzpicture}[scale=1.1]
\begin{scope}
\node[pn, label=below:$a_1$] (a1) at (1,0) {};
\node[pn, label=below:$a_2$] (a2) at (3,0) {};
\node[pn, label=below:$a_3$] (a3) at (5,0) {};
\node[pn, label=below:$a_{n+2}$] (an2) at (8,0) {};
\node[pn, label=above left:$b_1$] (b1) at (0,1.5) {};
\node[pn, label=left:$b_2$] (b2) at (2,1.5) {};
\node[pn, label=left:$b_3$] (b3) at (4,1.5) {};
\node[pn, label=right:$b_{n+2}$] (bn2) at (7,1.5) {};
\node[pn, label=right:$b_{n+3}$] (bn3) at (9,1.5) {};
\node[pn, label=above:$d$] (d) at (4,3.5) {};
\node[pn, label=right:$c$] (c) at (8,2.5) {};

\draw (a1)--(b1) (a1)--(b2) (a2)--(b2) (a2)--(b3) (a3)--(b3);
\draw (an2)--(bn2) (an2)--(bn3);
\draw (b2)--(d) (b3)--(d) (bn2)--(d) (c)--(d);
\draw (a3)--(5.5,0.75);
\draw (6.5,0.75)--(bn2);
\node at (6,0.75) {\Large \(\dots\)};
\end{scope}
\end{tikzpicture}%
}
\caption{The poset \(E_n\).}
\label{fig:kelly-En}
\end{figure}

The following two tables define \(W_{E_n}\).
\par\nopagebreak[4]
\noindent
{\small
\renewcommand{\arraystretch}{1.08}
\begin{minipage}[t]{0.49\textwidth}
\vspace{0pt}
\centering
\begin{tabular}{@{}clc@{}}
\toprule
Vertex & Critical pair & Range\\
\midrule
\(\alpha_i\) & \(\bigl((0,a_i),(q,c)\bigr)\) & \(1\leq i\leq n+2\)\\
\(\beta_i\) & \(\bigl((q,a_i),(p,d)\bigr)\) & \(1\leq i\leq n+2\)\\
\(\gamma_i\) & \(\bigl((0,c),(q,b_i)\bigr)\) & \(1\leq i\leq n+3\)\\
\(\pi_1\) & \(\bigl((p,a_1),(q,d)\bigr)\) & \\
\(\pi_2\) & \(\bigl((p,a_{n+2}),(q,d)\bigr)\) & \\
\(\pi_3\) & \(\bigl((q,c),(p,d)\bigr)\) & \\
\(\ell_1\) & \(\bigl((0,c),(p,b_1)\bigr)\) & \\
\(\ell_2\) & \(\bigl((0,a_{n+2}),(p,b_1)\bigr)\) & \\
\(\ell_3\) & \(\bigl((0,a_{n+2}),(q,b_1)\bigr)\) & \\
\(\ell_4\) & \(\bigl((0,c),(p,b_{n+3})\bigr)\) & \\
\(\ell_5\) & \(\bigl((0,a_1),(p,b_{n+3})\bigr)\) & \\
\(\ell_6\) & \(\bigl((0,a_1),(q,b_{n+3})\bigr)\) & \\
\(r_1\) & \(\bigl((0,b_1),(q,d)\bigr)\) & \\
\(r_2\) & \(\bigl((0,b_{n+3}),(q,d)\bigr)\) & \\
\bottomrule
\end{tabular}
\end{minipage}
\hfill
\begin{minipage}[t]{0.49\textwidth}
\vspace{0pt}
\centering
\begin{tabular}{@{}lll@{}}
\toprule
Vertices & Adjacent vertices & Range\\
\midrule
\(\alpha_1\) & \(\ell_1,\gamma_1,\gamma_2,\pi_3\) & \\
\(\alpha_i\) & \(\gamma_i,\gamma_{i+1},\pi_3\) & \(2\leq i\leq n+1\)\\
\(\alpha_{n+2}\) & \(\ell_4,\gamma_{n+2},\gamma_{n+3},\pi_3\) & \\
\(\beta_1\) & \(\pi_1,\pi_2,\gamma_1,\gamma_2,\ell_3\) & \\
\(\beta_i\) & \(\pi_1,\pi_2,\gamma_i,\gamma_{i+1}\) & \(2\leq i\leq n+1\)\\
\(\beta_{n+2}\) & \(\pi_1,\pi_2,\gamma_{n+2},\gamma_{n+3},\ell_6\) & \\
\(\gamma_1\) & \(r_1\) & \\
\(\gamma_{n+3}\) & \(r_2\) & \\
\(\pi_1\) & \(\pi_3,\ell_1,\ell_2\) & \\
\(\pi_2\) & \(\pi_3,\ell_4,\ell_5\) & \\
\(\ell_1\) & \(r_1\) & \\
\(\ell_2,\ell_3\) & \(\ell_5,\ell_6,r_1\) & \\
\(\ell_4,\ell_5,\ell_6\) & \(r_2\) & \\
\bottomrule
\end{tabular}
\end{minipage}
}

\begin{proposition}\label{prop:E-coloring}
For every \(n\geq0\),
\(W_{E_n}\subseteq\Gamma(\Vee\times E_n)\), and \(W_{E_n}\) is not
\(3\)-colorable.
\end{proposition}

\begin{proof}
\smallskip
\noindent\emph{Critical pairs.}
The minimal and maximal elements of \(E_n\) are
\[
 \Min(E_n)=\{a_i:1\leq i\leq n+2\}\cup\{c\},
 \qquad
 \Max(E_n)=\{b_1,b_{n+3},d\}.
\]
For \(\alpha_i\), condition~\eqref{eq:product-critical-first} follows
from
\[
 a_i\nleq c,\qquad
 \strictdown a_i=\varnothing,\qquad
 \strictup c=\{d\}\subseteq\uparrow a_i.
\]
For \(\gamma_i\), we have \(c\nleq b_i\) and
\(\strictdown c=\varnothing\).  The element \(b_i\) is maximal when
\(i=1\) or \(i=n+3\).  For \(2\leq i\leq n+2\),
\(\strictup b_i=\{d\}\subseteq\uparrow c\).  Thus
\eqref{eq:product-critical-first} also holds for every \(\gamma_i\).
For each of \(\ell_1,\ldots,\ell_6\), the first \(E_n\)-coordinate
is minimal, the second is maximal, and the two are incomparable.
Hence \eqref{eq:product-critical-first} holds for these six vertices.
For \(r_1\) and \(r_2\), condition~\eqref{eq:product-critical-first}
follows from
\[
\begin{array}{llll}
 r_1:& b_1\nleq d,&
 \strictdown b_1=\{a_1\}\subseteq\downarrow d,&
 \strictup d=\varnothing,\\
 r_2:& b_{n+3}\nleq d,&
 \strictdown b_{n+3}=\{a_{n+2}\}\subseteq\downarrow d,&
 \strictup d=\varnothing.
\end{array}
\]
By Lemma~\ref{lem:product-critical}\textup{(i)},
\(\alpha_i,\gamma_i,\ell_1,\ldots,\ell_6,r_1,r_2\) are critical pairs.
Since \(a_i\leq d\) and \(c\leq d\),
Lemma~\ref{lem:product-critical}\textup{(ii)} shows that
\(\beta_i,\pi_1,\pi_2,\pi_3\) are critical pairs.

\smallskip
\noindent\emph{Edges.}
For every adjacency in the right table, the two \(E_n\)-comparisons
in~\eqref{eq:product-edge} are
\[
 a_i\leq b_i,\qquad
 a_i\leq b_{i+1},\qquad
 c\leq d,\qquad
 a_i\leq d,
\]
or equalities.  These comparisons follow from the Hasse diagram of
\(E_n\).  Hence
\[
  W_{E_n}\subseteq\Gamma(\Vee\times E_n).
\]

Suppose that \(\varphi\) is a proper \(3\)-coloring of \(W_{E_n}\)
with colors \(0,1,2\).

\medskip
\noindent\textbf{Claim 1.}
\(\varphi(\pi_1)=\varphi(\pi_2)\).

\smallskip
Assume otherwise.  We may take $\varphi(\pi_1)=0$ and $\varphi(\pi_2)=1$.
Every \(\beta_i\) is adjacent to both \(\pi_1\) and \(\pi_2\).
Thus \(\varphi(\beta_i)=2\) for \(1\leq i\leq n+2\).
The vertex \(\pi_3\) is also adjacent to both, so
\(\varphi(\pi_3)=2\).  The right table gives
\(\gamma_1\sim\beta_1\) and
\(\gamma_i\sim\beta_{i-1}\) for \(2\leq i\leq n+3\).  Hence
\(\varphi(\gamma_i)\in\{0,1\}\) for \(1\leq i\leq n+3\).

For \(1\leq i\leq n+2\), the vertex \(\alpha_i\) is adjacent to
\(\gamma_i,\gamma_{i+1},\pi_3\).  The vertices \(\gamma_i\) and
\(\gamma_{i+1}\) have the same color.  Otherwise \(\alpha_i\) would
have neighbors of all three colors.  Thus
\[
  \varphi(\gamma_1)=\cdots=\varphi(\gamma_{n+3})=t
\]
for some \(t\in\{0,1\}\).

Suppose first that \(t=0\).
\[
\begin{aligned}
 \alpha_1\sim\gamma_1,\pi_3
   &\Longrightarrow \varphi(\alpha_1)=1,\\
 \ell_1\sim\alpha_1,\pi_1
   &\Longrightarrow \varphi(\ell_1)=2,\\
 r_1\sim\ell_1,\gamma_1
   &\Longrightarrow \varphi(r_1)=1,\\
 \ell_2\sim r_1,\pi_1
   &\Longrightarrow \varphi(\ell_2)=2,\\
 \ell_3\sim r_1,\beta_1
   &\Longrightarrow \varphi(\ell_3)=0.
\end{aligned}
\]
The neighbors \(\ell_3,\pi_2,\ell_2\) of \(\ell_5\) have colors
\(0,1,2\), a contradiction.

Suppose that \(t=1\).
\[
\begin{aligned}
 \alpha_{n+2}\sim\gamma_{n+2},\pi_3
   &\Longrightarrow \varphi(\alpha_{n+2})=0,\\
 \ell_4\sim\alpha_{n+2},\pi_2
   &\Longrightarrow \varphi(\ell_4)=2,\\
 r_2\sim \ell_4,\gamma_{n+3}
   &\Longrightarrow \varphi(r_2)=0,\\
 \ell_6\sim r_2,\beta_{n+2}
   &\Longrightarrow \varphi(\ell_6)=1.
\end{aligned}
\]

Both \(\ell_2\) and \(\ell_3\) are adjacent to \(\ell_6\), so neither
has color \(1\).  Since \(r_1\sim\gamma_1\), we also have
\(\varphi(r_1)\neq1\).  If \(\varphi(r_1)=0\), then
\(\varphi(\ell_3)=1\), since \(\ell_3\sim r_1,\beta_1\).  If
\(\varphi(r_1)=2\), then \(\varphi(\ell_2)=1\), since
\(\ell_2\sim r_1,\pi_1\).  Both conclusions are impossible.  This proves
Claim~1.

By Claim~1, we may take
\[
  \varphi(\pi_1)=\varphi(\pi_2)=0,
  \qquad
  \varphi(\pi_3)=1.
\]

\medskip
\noindent\textbf{Claim 2.}
\(\varphi(\gamma_i)\in\{0,1\}\) for every \(1\leq i\leq n+3\).

\smallskip
Suppose that \(\varphi(\gamma_j)=2\) for some \(j\).  Fix
\(1\leq i\leq n+2\).  
If \(\varphi(\gamma_i)=2\), then \(\varphi(\beta_i)=1\), since \(\beta_i\sim\gamma_i,\pi_1\).  
The edge
\(\beta_i\sim\gamma_{i+1}\) excludes color \(1\) from
\(\gamma_{i+1}\).  If \(\varphi(\gamma_{i+1})=0\), then \(\alpha_i\)
would have neighbors of all three colors.  Thus
\(\varphi(\gamma_{i+1})=2\).  The reverse implication follows in the
same way.  Therefore
\[
 \varphi(\gamma_i)=2
 \quad\Longleftrightarrow\quad
 \varphi(\gamma_{i+1})=2.
\]
The equivalence holds for every \(i\), so
\[
  \varphi(\gamma_i)=2 \qquad (1\leq i\leq n+3).
\]
The adjacencies \(\beta_i\sim\pi_1,\gamma_i\) imply
\[
  \varphi(\beta_i)=1
  \qquad (1\leq i\leq n+2).
\]

Since \(r_1\sim\gamma_1\), we have
\(\varphi(r_1)\in\{0,1\}\).  If \(\varphi(r_1)=0\), then
\(\varphi(\ell_3)=2\), since \(\ell_3\sim r_1,\beta_1\).  If
\(\varphi(r_1)=1\), then \(\varphi(\ell_2)=2\), since
\(\ell_2\sim r_1,\pi_1\).  Thus at least one of \(\ell_2,\ell_3\) has
color \(2\).

The vertex \(\ell_6\) is adjacent to \(\beta_{n+2}\), which has color
\(1\), and to a vertex of color \(2\) in
\(\{\ell_2,\ell_3\}\).  Thus \(\varphi(\ell_6)=0\).  Each of
\(\ell_2,\ell_3,r_2\) is adjacent to \(\ell_6\), so none has color
\(0\).  Since
\(r_2\sim\gamma_{n+3}\), we have \(\varphi(r_2)=1\).  The neighbors of
\(\ell_5\) include \(\pi_2\) of color \(0\), \(r_2\) of color \(1\),
and a vertex of color \(2\) in \(\{\ell_2,\ell_3\}\).  Thus \(\ell_5\)
has neighbors of all three colors, a contradiction.  This proves Claim~2.

\pagebreak[4]
\medskip
\noindent\textbf{Claim 3.}
One of \(\ell_2,\ell_3\) has color \(1\).

\smallskip
Since \(\ell_1\sim\pi_1\), we have
\(\varphi(\ell_1)\neq0\).  Suppose \(\varphi(r_1)=1\).  By Claim~2 and
\(r_1\sim\gamma_1\), we have \(\varphi(\gamma_1)=0\).  Since
\(\ell_1\sim\pi_1,r_1\), we have \(\varphi(\ell_1)=2\).  Thus
\(\alpha_1\) has neighbors \(\gamma_1,\pi_3,\ell_1\) of colors
\(0,1,2\), a contradiction.

Hence \(\varphi(r_1)\in\{0,2\}\).  If
\(\varphi(r_1)=2\), the neighbors \(\pi_1,r_1\) of
\(\ell_2\) have colors \(0,2\), and \(\varphi(\ell_2)=1\).  If
\(\varphi(r_1)=0\), Claim~2 and \(r_1\sim\gamma_1\) imply
\(\varphi(\gamma_1)=1\).  The adjacencies
\(\beta_1\sim\gamma_1,\pi_1\) and
\(\ell_3\sim r_1,\beta_1\) imply
\(\varphi(\beta_1)=2\) and \(\varphi(\ell_3)=1\).
This proves Claim~3.

\medskip
\noindent\textbf{Claim 4.}
\(\varphi(r_2)\neq1\).

\smallskip
Assume otherwise.  Claim~2 and
\(r_2\sim\gamma_{n+3}\) imply \(\varphi(\gamma_{n+3})=0\).  Since
\(\ell_4\sim r_2,\pi_2\), we have \(\varphi(\ell_4)=2\).  The neighbors
\(\gamma_{n+3},\pi_3,\ell_4\) of \(\alpha_{n+2}\) use colors
\(0,1,2\), a contradiction.  This proves Claim~4.

By Claim~3, one of \(\ell_2,\ell_3\) has color \(1\); both vertices are
adjacent to \(\ell_6\).  Since
\(\beta_{n+2}\sim\pi_1\), we have
\(\varphi(\beta_{n+2})\in\{1,2\}\).  If
\(\varphi(\beta_{n+2})=2\), then \(\ell_6\) has a neighbor of
color \(2\), namely \(\beta_{n+2}\).  Suppose instead that
\(\varphi(\beta_{n+2})=1\).  Claim~2 and
\(\beta_{n+2}\sim\gamma_{n+3}\) imply
\(\varphi(\gamma_{n+3})=0\).  Since \(\ell_4\sim\pi_2\), we have
\(\varphi(\ell_4)\neq0\).  If \(\varphi(\ell_4)=2\), the neighbors
\(\gamma_{n+3},\pi_3,\ell_4\) of \(\alpha_{n+2}\) have colors
\(0,1,2\), a contradiction.  Hence \(\varphi(\ell_4)=1\), and the
neighbors \(\gamma_{n+3},\ell_4\) of \(r_2\) have colors \(0,1\).
Thus \(\varphi(r_2)=2\).  In this case, \(r_2\) is a neighbor of
\(\ell_6\) with color \(2\).

In both cases, \(\ell_6\) has neighbors of colors \(1\) and \(2\), so
\(\varphi(\ell_6)=0\).  Each of
\(\ell_2,\ell_3,r_2,\beta_{n+2}\) is adjacent to \(\ell_6\), so none
has color \(0\).  By Claim~4,
\(\varphi(r_2)=2\).  By Claim~3, one of \(\ell_2,\ell_3\) has color
\(1\).  This vertex, \(\pi_2\), and \(r_2\) are three differently
colored neighbors of \(\ell_5\), a contradiction.
Thus \(W_{E_n}\) is not
\(3\)-colorable.
\end{proof}

\subsection{The family \texorpdfstring{\(F_n\)}{F-n}}

\begin{figure}[!htbp]
\centering
\resizebox{0.70\textwidth}{!}{%
\begin{tikzpicture}[scale=1.1]
\node[pn, label=below:$c$] (c) at (3.5,-1.5) {};
\node[pn, label=below:$a_1$] (a1) at (0.5,0) {};
\node[pn, label=below:$a_2$] (a2) at (1.5,0) {};
\node[pn, label=left:$a_{n+1}$] (an1) at (4,0) {};
\node[pn, label=right:$a_{n+2}$] (an2) at (5,0) {};
\node[pn, label=left:$b_1$] (b1) at (0,2) {};
\node[pn, label=left:$b_2$] (b2) at (1,2) {};
\node[pn, label=left:$b_3$] (b3) at (2,2) {};
\node[pn, label=right:$b_{n+2}$] (bn2) at (4.5,2) {};
\node[pn, label=right:$d$] (d) at (8,1) {};
\node[pn, label=above:$e$] (e) at (3.5,3.5) {};

\draw (a1)--(b1) (a1)--(b2) (a2)--(b2) (a2)--(b3);
\draw (an1)--(bn2) (an2)--(bn2);
\draw (a1)--(c) (a2)--(c) (an1)--(c) (d)--(c);
\draw (b2)--(e) (b3)--(e) (bn2)--(e) (d)--(e);
\draw (b3)--(2.5,1);
\draw (3.5,1)--(an1);
\node at (3,1) {\Large \(\dots\)};
\end{tikzpicture}%
}
\caption{The poset \(F_n\).}
\label{fig:kelly-Fn}
\end{figure}
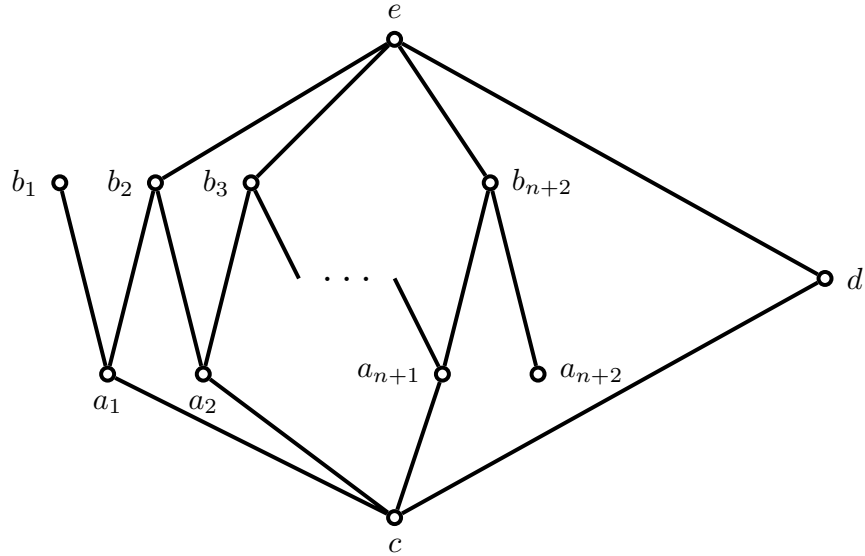

The following two tables define \(W_{F_n}\).
\par\nopagebreak[4]
\noindent
{\small
\renewcommand{\arraystretch}{1.08}
\begin{minipage}[t]{0.49\textwidth}
\vspace{0pt}
\centering
\begin{tabular}{@{}clc@{}}
\toprule
Vertex & Critical pair & Range\\
\midrule
\(\alpha_i\) & \(\bigl((0,a_i),(q,d)\bigr)\) & \(1\leq i\leq n+2\)\\
\(\beta_i\) & \(\bigl((0,d),(q,b_i)\bigr)\) & \(1\leq i\leq n+2\)\\
\(\pi_1\) & \(\bigl((q,c),(p,e)\bigr)\) & \\
\(\pi_2\) & \(\bigl((q,c),(p,b_1)\bigr)\) & \\
\(\pi_3\) & \(\bigl((q,a_{n+2}),(p,e)\bigr)\) & \\
\(\pi_4\) & \(\bigl((p,c),(q,e)\bigr)\) & \\
\(\ell_1\) & \(\bigl((0,a_{n+2}),(p,b_1)\bigr)\) & \\
\(\ell_2\) & \(\bigl((0,a_{n+2}),(q,b_1)\bigr)\) & \\
\(r_1\) & \(\bigl((0,b_1),(q,e)\bigr)\) & \\
\(r_2\) & \(\bigl((0,c),(q,a_{n+2})\bigr)\) & \\
\bottomrule
\end{tabular}
\end{minipage}
\hfill
\begin{minipage}[t]{0.49\textwidth}
\vspace{0pt}
\centering
\begin{tabular}{@{}lll@{}}
\toprule
Vertices & Adjacent vertices & Range\\
\midrule
\(\alpha_1\) & \(\beta_1,\beta_2,\pi_1,\pi_2\) & \\
\(\alpha_i\) & \(\beta_i,\beta_{i+1},\pi_1\) & \(2\leq i\leq n+1\)\\
\(\alpha_{n+2}\) & \(\beta_{n+2},\pi_1,r_2\) & \\
\(\beta_1\) & \(\pi_1,r_1\) & \\
\(\beta_i\) & \(\pi_1\) & \(2\leq i\leq n+1\)\\
\(\beta_{n+2}\) & \(\pi_1,\pi_3\) & \\
\(\pi_1\) & \(\pi_4,\ell_2\) & \\
\(\pi_2\) & \(\pi_4,r_1\) & \\
\(\pi_3\) & \(\pi_4,r_2\) & \\
\(\pi_4\) & \(\ell_1\) & \\
\(\ell_1,\ell_2\) & \(r_1,r_2\) & \\
\bottomrule
\end{tabular}
\end{minipage}
}

\begin{proposition}\label{prop:F-coloring}
For every \(n\geq0\),
\(W_{F_n}\subseteq\Gamma(\Vee\times F_n)\), and \(W_{F_n}\) is not
\(3\)-colorable.
\end{proposition}

\begin{proof}
\smallskip
\noindent\emph{Critical pairs.}
The minimal and maximal elements of \(F_n\) are
\[
 \Min(F_n)=\{c,a_{n+2}\},
 \qquad
 \Max(F_n)=\{b_1,e\}.
\]
For \(1\leq i\leq n+1\), condition~\eqref{eq:product-critical-first}
for \(\alpha_i\) follows from
\[
 a_i\nleq d,\qquad
 \strictdown a_i=\{c\}\subseteq\downarrow d,\qquad
 \strictup d=\{e\}\subseteq\uparrow a_i.
\]
For \(\alpha_{n+2}\), the same condition follows from
\[
 a_{n+2}\nleq d,\qquad
 \strictdown a_{n+2}=\varnothing,\qquad
 \strictup d=\{e\}\subseteq\uparrow a_{n+2}.
\]
For \(\beta_i\), we have
\(d\nleq b_i\) and
\(\strictdown d=\{c\}\subseteq\downarrow b_i\).
If \(i=1\), then \(b_i\) is maximal.  If \(2\leq i\leq n+2\), then
\(\strictup b_i=\{e\}\subseteq\uparrow d\).
Thus condition~\eqref{eq:product-critical-first} also holds for every
\(\beta_i\).
For \(\ell_1\) and \(\ell_2\), the first \(F_n\)-coordinate is
minimal, the second is maximal, and the two are incomparable.
Hence condition~\eqref{eq:product-critical-first} holds for these two
vertices.
For \(r_1,r_2\), condition~\eqref{eq:product-critical-first}
follows from
\[
\begin{array}{llll}
 r_1:& b_1\nleq e,&
 \strictdown b_1=\{c,a_1\}\subseteq\downarrow e,&
 \strictup e=\varnothing,\\
 r_2:& c\nleq a_{n+2},&
 \strictdown c=\varnothing,&
 \strictup a_{n+2}=\{b_{n+2},e\}\subseteq\uparrow c.
\end{array}
\]
By Lemma~\ref{lem:product-critical}\textup{(i)},
\(\alpha_i,\beta_i,\ell_1,\ell_2,r_1,r_2\) are critical pairs.
Since \(c\leq b_1\), \(c\leq e\), and \(a_{n+2}\leq e\),
Lemma~\ref{lem:product-critical}\textup{(ii)} shows that
\(\pi_1,\pi_2,\pi_3,\pi_4\) are critical pairs.

\smallskip
\noindent\emph{Edges.}
By the Hasse diagram of \(F_n\), the two \(F_n\)-comparisons in
\eqref{eq:product-edge} hold for every adjacency in the right table.
Hence
\[
  W_{F_n}\subseteq\Gamma(\Vee\times F_n).
\]

Suppose that \(\varphi\) is a proper \(3\)-coloring of \(W_{F_n}\)
with colors \(0,1,2\).

\medskip
\noindent\textbf{Claim 1.}
All \(\beta_i\) have the same color.  All \(\alpha_i\) have the same
color.  These two colors and \(\varphi(\pi_1)\) are distinct.

\smallskip
Fix \(1\leq i\leq n+1\).  The vertices \(\alpha_i\) and \(\pi_1\)
are adjacent.  Both \(\beta_i\) and \(\beta_{i+1}\) are adjacent to
these two vertices.  Hence
\(\varphi(\beta_i)=\varphi(\beta_{i+1})\).  Applying this equality for
all \(i\) gives
\[
  \varphi(\beta_1)=\cdots=\varphi(\beta_{n+2}).
\]
This color differs from \(\varphi(\pi_1)\).  For
\(1\leq i\leq n+1\), we have \(\alpha_i\sim\beta_i,\pi_1\).
Thus \(\alpha_1,\ldots,\alpha_{n+1}\) have the third color.  The
adjacencies \(\alpha_{n+2}\sim\beta_{n+2},\pi_1\) show that
\(\alpha_{n+2}\) has the same third color.  This proves Claim~1.

By Claim~1, we may take
\[
  \varphi(\beta_i)=0 \quad (1\leq i\leq n+2),\qquad
  \varphi(\alpha_i)=1 \quad (1\leq i\leq n+2),\qquad
  \varphi(\pi_1)=2.
\]
Since \(\pi_4\sim\pi_1\), we have
\(\varphi(\pi_4)\in\{0,1\}\).

\medskip
\noindent\textbf{Claim 2.}
\(\varphi(\pi_4)=1\).

\smallskip
Suppose that \(\varphi(\pi_4)=0\).
\[
\begin{aligned}
 \pi_2\sim \pi_4,\alpha_1
   &\Longrightarrow \varphi(\pi_2)=2,\\
 r_1\sim\pi_2,\beta_1
   &\Longrightarrow \varphi(r_1)=1,\\
 \ell_2\sim r_1,\pi_1
   &\Longrightarrow \varphi(\ell_2)=0,\\
 r_2\sim\ell_2,\alpha_{n+2}
   &\Longrightarrow \varphi(r_2)=2.
\end{aligned}
\]
The neighbors \(\pi_4,r_1,r_2\) of \(\ell_1\) have colors
\(0,1,2\), a contradiction.  This proves Claim~2.

By Claim~2, \(\varphi(\pi_4)=1\).  Then
\[
\begin{aligned}
 \pi_3\sim \pi_4,\beta_{n+2}
   &\Longrightarrow \varphi(\pi_3)=2,\\
 r_2\sim\alpha_{n+2},\pi_3
   &\Longrightarrow \varphi(r_2)=0,\\
 \ell_1\sim \pi_4,r_2
   &\Longrightarrow \varphi(\ell_1)=2,\\
 r_1\sim\beta_1,\ell_1
   &\Longrightarrow \varphi(r_1)=1.
\end{aligned}
\]
The neighbors \(r_2,r_1,\pi_1\) of \(\ell_2\) have colors
\(0,1,2\), a contradiction.  Hence \(W_{F_n}\) is not
\(3\)-colorable.
\end{proof}

\subsection{The families \texorpdfstring{\(G_n\) and \(J_n\)}{G-n and J-n}}

\begin{figure}[!htbp]
\centering
\begin{minipage}[c]{0.47\textwidth}
\centering
\resizebox{!}{0.36\textheight}{%
\begin{tikzpicture}[scale=1.1]
    \node[pn, label=below:$a_1$] (a1) at (0,0) {}; 
    \node[pn, label=left:$a_2$] (a2) at (0,1) {}; 
    \node[pn, label=left:$a_3$] (a3) at (0,2) {}; 
    \node[pn, label=left:$a_4$] (a4) at (0,3) {};
    \node[pn, label=right:$b_1$] (b1) at (3,0.5) {}; 
    \node[pn, label=right:$b_2$] (b2) at (3,1.5) {}; 
    \node[pn, label=right:$b_n$] (bn) at (3,4.5) {}; 
    \node[pn, label=right:$b_{n+1}$] (bn1) at (3,5.5) {};
    \node[pn, label=left:$b_{n+2}$] (bn2) at (3,6.5) {}; 
    \node[pn, label=right:$b_{n+3}$] (bn3) at (3,7.5) {};
    \node[pn, label=left:$a_{n+2}$] (an2) at (0,6) {}; 
    \node[pn, label=left:$a_{n+3}$] (an3) at (0,7) {};
    \node[pn, label=right:$c$] (c) at (5,1.5) {};
    
    \draw (a1)--(a2) (a2)--(a3) (a3)--(a4);
    \draw (b1)--(b2);
    \draw (an2)--(an3);
    \draw (bn)--(bn1) (bn1)--(bn2) (bn2)--(bn3);
    \draw (a1)--(b2); 
    \draw (b1)--(a3); 
    \draw (b2)--(a4); 
    \draw (bn)--(an2);  
    \draw (bn1)--(an3); 
    \draw (an2)--(bn3);
    \draw (a1)--(c) (c)--(bn3);
    
    \draw (a2) -- (2,2);
    \draw[dashed] (2,2) -- (3,2.5); 
    \draw (a3) -- (1, 2.5);
    \draw[dashed] (1, 2.5) -- (2, 3);
    \draw (bn2) -- (1,5.5);
    \draw[dashed] (1,5.5) -- (0,5);
    \draw (bn1) -- (2, 5);
    \draw[dashed] (2, 5) -- (1, 4.5); 
    
    \draw[dashed] (b2) -- (3,2.5);
    \draw[dashed] (a4) -- (0, 3.8);
    \draw[dashed] (a4) -- (1.2, 3.6); 
    \draw[dashed] (an2) -- (0,5);
    \draw[dashed] (bn) -- (3, 3.7); 
    \draw[dashed] (bn) -- (1.8, 3.9); 
    \node[title] at (1.5,-1) {$G_n$};
\end{tikzpicture}
}
\end{minipage}\hfill
\begin{minipage}[c]{0.47\textwidth}
\centering
\resizebox{!}{0.36\textheight}{%
\begin{tikzpicture}[scale=1.1]
    \node[pn, label=below:$a_1$] (a1) at (0,0) {}; 
    \node[pn, label=left:$a_2$] (a2) at (0,1) {}; 
    \node[pn, label=left:$a_3$] (a3) at (0,2) {}; 
    \node[pn, label=left:$a_4$] (a4) at (0,3) {};
    \node[pn, label=right:$b_1$] (b1) at (3,0.5) {}; 
    \node[pn, label=right:$b_2$] (b2) at (3,1.5) {}; 
    \node[pn, label=right:$b_n$] (bn) at (3,4.5) {}; 
    \node[pn, label=right:$b_{n+1}$] (bn1) at (3,5.5) {};
    \node[pn, label=left:$b_{n+2}$] (bn2) at (3,6.5) {}; 
    \node[pn, label=right:$b_{n+3}$] (bn3) at (3,7.5) {};
    \node[pn, label=left:$a_{n+2}$] (an2) at (0,6) {}; 
    \node[pn, label=left:$a_{n+3}$] (an3) at (0,7) {};
    \node[pn, label=right:$c$] (c) at (6,1.5) {}; 
    \node[pn, label=right:$d$] (d) at (6,3.5) {};
    
    \draw (a1)--(a2) (a2)--(a3) (a3)--(a4) (b1)--(b2);
    \draw (an2)--(an3) (bn)--(bn1) (bn1)--(bn2) (bn2)--(bn3);
    \draw (a1)--(b2) (b1)--(a3) (b2)--(a4);
    \draw (bn)--(an2) (bn1)--(an3) (an2)--(bn3);
    \draw (a1)--(d) (c)--(d) (c)--(bn3);
    
    \draw (a2) -- (2,2);
    \draw[dashed] (2,2) -- (3,2.5); 
    \draw (a3) -- (1, 2.5);
    \draw[dashed] (1, 2.5) -- (2, 3);
    \draw (bn2) -- (1,5.5);
    \draw[dashed] (1,5.5) -- (0,5);
    \draw (bn1) -- (2, 5);
    \draw[dashed] (2, 5) -- (1, 4.5); 
    
    \draw[dashed] (b2) -- (3,2.5);
    \draw[dashed] (a4) -- (0, 3.8);
    \draw[dashed] (a4) -- (1.2, 3.6); 
    \draw[dashed] (an2) -- (0,5);
    \draw[dashed] (bn) -- (3, 3.7); 
    \draw[dashed] (bn) -- (1.8, 3.9); 
    
    \node[title] at (1.5,-1) {$J_n$};
\end{tikzpicture}
}
\end{minipage}
\caption{The posets \(G_n\) and \(J_n\).}
\label{fig:kelly-GJ}
\end{figure}
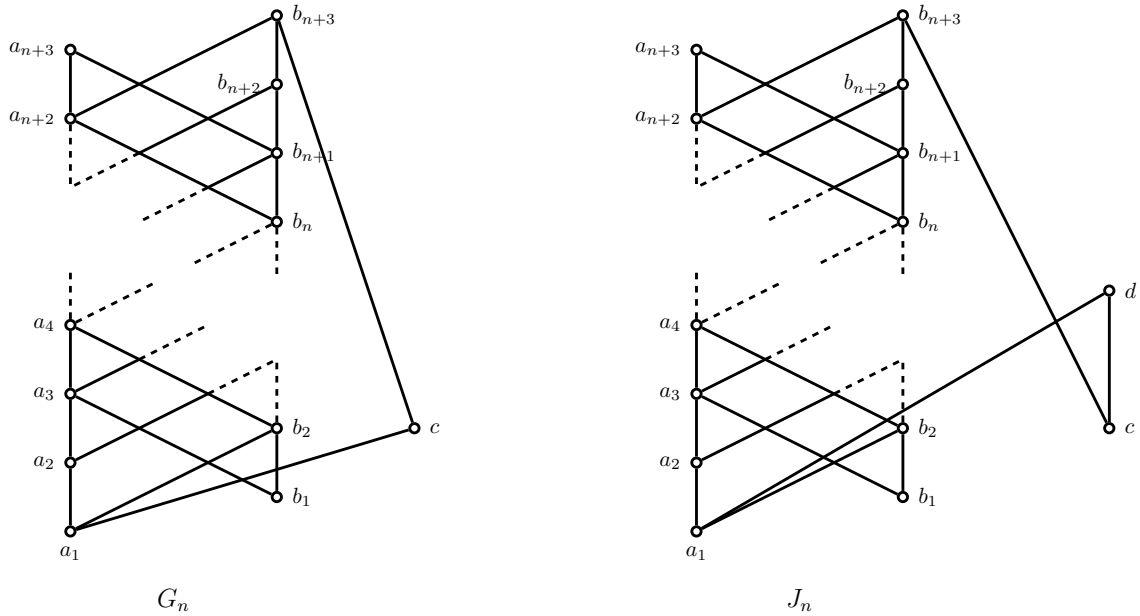

Put \(z=c\) for \(W_{G_n}\) and \(z=d\) for \(W_{J_n}\).  The
following two tables define the two graphs.
\par\nopagebreak[4]
\noindent
{\small
\renewcommand{\arraystretch}{1.08}
\begin{minipage}[t]{0.49\textwidth}
\vspace{0pt}
\centering
\begin{tabular}{@{}clc@{}}
\toprule
Vertex & Critical pair & Range\\
\midrule
\(\alpha_i\) & \(\bigl((0,a_i),(q,b_i)\bigr)\) & \(1\leq i\leq n+3\)\\
\(\beta_i\) & \(\bigl((0,b_i),(q,a_{i+1})\bigr)\) & \(1\leq i\leq n+2\)\\
\(\pi_1\) & \(\bigl((q,a_1),(p,a_{n+3})\bigr)\) & \\
\(\pi_2\) & \(\bigl((q,b_1),(p,b_{n+3})\bigr)\) & \\
\(\pi_3\) & \(\bigl((q,a_1),(p,b_{n+3})\bigr)\) & \\
\(r_1\) & \(\bigl((0,c),(q,a_{n+3})\bigr)\) & \\
\(r_2\) & \(\bigl((0,b_1),(q,z)\bigr)\) & \\
\bottomrule
\end{tabular}
\end{minipage}\hfill
\begin{minipage}[t]{0.49\textwidth}
\vspace{0pt}
\centering
\begin{tabular}{@{}lll@{}}
\toprule
Vertices & Adjacent vertices & Range\\
\midrule
\(\alpha_1\) & \(\beta_1,r_2,\pi_2\) & \\
\(\alpha_i\) & \(\beta_{i-1},\beta_i,\pi_1,\pi_2,\pi_3\) &
  \(2\leq i\leq n+2\)\\
\(\alpha_{n+3}\) & \(\beta_{n+2},r_1,\pi_1\) & \\
\(\beta_1\) & \(\pi_1,\pi_3\) & \\
\(\beta_i\) & \(\pi_1,\pi_2,\pi_3\) & \(2\leq i\leq n+1\)\\
\(\beta_{n+2}\) & \(\pi_2,\pi_3\) & \\
\(r_1\) & \(r_2\) & \\
\(\pi_1\) & \(r_2\) & \\
\(\pi_2\) & \(r_1\) & \\
\(\pi_3\) & \(r_1,r_2\) & \\

\bottomrule
\end{tabular}
\end{minipage}
}

In each graph, the sequence
\[
 \alpha_1,\beta_1,\alpha_2,\beta_2,\ldots,
 \alpha_{n+2},\beta_{n+2},\alpha_{n+3},r_1,r_2,\alpha_1
\]
is a cycle of length \(2n+7\):
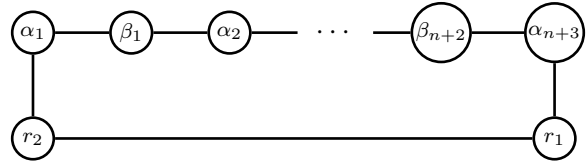
\begin{center}
\begin{tikzpicture}[
  every path/.style={line width=1pt},
  wg/.style={circle,draw,fill=white,minimum size=5.5mm,inner sep=0.6pt,
    line width=1pt,font=\scriptsize}
]
  \node[wg] (a1) at (0,1.2) {$\alpha_1$};
  \node[wg] (b1) at (1.3,1.2) {$\beta_1$};
  \node[wg] (a2) at (2.6,1.2) {$\alpha_2$};
  \node at (4.0,1.2) {$\cdots$};
  \node[wg] (bm) at (5.4,1.2) {$\beta_{n+2}$};
  \node[wg] (am) at (6.9,1.2) {$\alpha_{n+3}$};
  \node[wg] (r1) at (6.9,-0.2) {$r_1$};
  \node[wg] (r2) at (0,-0.2) {$r_2$};
  \draw (a1)--(b1)--(a2);
  \draw (a2)--(3.5,1.2) (4.5,1.2)--(bm);
  \draw (bm)--(am)--(r1)--(r2)--(a1);
\end{tikzpicture}
\end{center}
Every vertex of the cycle is adjacent to at least one of
\(\pi_1\) and \(\pi_2\).  The coloring argument uses \(\pi_3\) only
when \(n=0\).

\begin{proposition}\label{prop:GJ-coloring}
For every \(n\geq0\),
\[
  W_{G_n}\subseteq\Gamma(\Vee\times G_n),
  \qquad
  W_{J_n}\subseteq\Gamma(\Vee\times J_n),
\]
and neither \(W_{G_n}\) nor \(W_{J_n}\) is \(3\)-colorable.
\end{proposition}

\begin{proof}
\smallskip
\noindent\emph{Critical pairs.}  Fix \(R\in\{G_n,J_n\}\).
The minimal and maximal elements are
\[
\begin{aligned}
 \Min(G_n)&=\{a_1,b_1\},&
 \Max(G_n)&=\{a_{n+3},b_{n+3}\},\\
 \Min(J_n)&=\{a_1,b_1,c\},&
 \Max(J_n)&=\{a_{n+3},b_{n+3},d\}.
\end{aligned}
\]
In both \(G_n\) and \(J_n\),
\[
 a_i\leq_R b_j\iff j\geq i+1,
 \qquad
 b_i\leq_R a_j\iff j\geq i+2.
\]
By these relations, condition~\eqref{eq:product-critical-first} holds
for \(\alpha_i\) and \(\beta_i\):
\[
\begin{array}{llll}
 \alpha_i:& a_i\nleq_R b_i,&
 \strictdown a_i\subseteq\downarrow b_i,&
 \strictup b_i\subseteq\uparrow a_i,\\
 \beta_i:& b_i\nleq_R a_{i+1},&
 \strictdown b_i\subseteq\downarrow a_{i+1},&
 \strictup a_{i+1}\subseteq\uparrow b_i.
\end{array}
\]
For \(G_n\), condition~\eqref{eq:product-critical-first} for
\(r_1,r_2\) follows from
\[
\begin{array}{llll}
 r_1:& c\nleq a_{n+3},&
 \strictdown c=\{a_1\}\subseteq\downarrow a_{n+3},&
 \strictup a_{n+3}=\varnothing,\\
 r_2:& b_1\nleq c,&
 \strictdown b_1=\varnothing,&
 \strictup c=\{b_{n+3}\}\subseteq\uparrow b_1.
\end{array}
\]
For \(J_n\), the two coordinate pairs used by \(r_1,r_2\) are
\(c\parallel a_{n+3}\) and \(b_1\parallel d\).  Each is an
incomparable pair in \(\Min(J_n)\times\Max(J_n)\), so
condition~\eqref{eq:product-critical-first} holds.
By Lemma~\ref{lem:product-critical}\textup{(i)},
\(\alpha_i,\beta_i,r_1,r_2\) are critical pairs.  Since
\(a_1\leq_R a_{n+3}\), \(b_1\leq_R b_{n+3}\), and
\(a_1\leq_R b_{n+3}\),
Lemma~\ref{lem:product-critical}\textup{(ii)} shows that
\(\pi_1,\pi_2,\pi_3\) are critical pairs.

\smallskip
\noindent\emph{Edges.}
By the two Hasse diagrams, all comparisons in
\eqref{eq:product-edge} hold for the adjacencies in the right table.
Hence
\[
  W_{G_n}\subseteq\Gamma(\Vee\times G_n),
  \qquad
  W_{J_n}\subseteq\Gamma(\Vee\times J_n).
\]

Let \(W\) be either \(W_{G_n}\) or \(W_{J_n}\), and suppose that
\(\varphi\) is a proper \(3\)-coloring of \(W\) with colors \(0,1,2\).

\medskip
\noindent\textbf{Claim.}
\(\varphi(\pi_1)\neq\varphi(\pi_2)\).

\smallskip
Assume that
\(\varphi(\pi_1)=\varphi(\pi_2)\).  Every vertex of the odd cycle is
adjacent to at least one of \(\pi_1,\pi_2\), so no vertex of the cycle
can receive their common color.  Hence the cycle uses only the other two
colors.  This would be a proper \(2\)-coloring of an odd cycle, which is
impossible.
This proves the claim.

By the claim, we may take
\[
 \varphi(\pi_1)=0,
 \qquad
 \varphi(\pi_2)=1.
\]
Then
\[
\begin{aligned}
 \alpha_2\sim\pi_1,\pi_2
   &\Longrightarrow \varphi(\alpha_2)=2,\\
 \beta_1\sim\pi_1,\alpha_2
   &\Longrightarrow \varphi(\beta_1)=1.
\end{aligned}
\]
If \(n\geq1\), the neighbors \(\pi_1,\pi_2,\alpha_2\) of
\(\beta_2\) have colors \(0,1,2\), a contradiction.

Let \(n=0\).  We have
\[
 \beta_2\sim\pi_2,\alpha_2
 \quad\Longrightarrow\quad
 \varphi(\beta_2)=0.
\]
The neighbors \(\beta_2,\beta_1,\alpha_2\) of \(\pi_3\) have colors
\(0,1,2\), again a contradiction.
\end{proof}

\subsection{The family \texorpdfstring{\(H_n\)}{H-n}}

\begin{figure}[htbp]
\centering
\resizebox{!}{0.38\textheight}{%
\begin{tikzpicture}[scale=1.1]
    \node[pn, label=below:$a_1$] (a1) at (0,0) {}; 
    \node[pn, label=left:$a_2$] (a2) at (0,1) {}; 
    \node[pn, label=left:$a_3$] (a3) at (0,2) {}; 
    \node[pn, label=left:$a_4$] (a4) at (0,3) {};
    \node[pn, label=right:$b_1$] (b1) at (3,0) {}; 
    \node[pn, label=right:$b_2$] (b2) at (3,1) {}; 
    \node[pn, label=right:$b_3$] (b3) at (3,2) {}; 
    \node[pn, label=right:$b_n$] (bn) at (3,5.25) {}; 
    \node[pn, label=right:$b_{n+1}$] (bn1) at (3,6.25) {};
    \node[pn, label=left:$b_{n+2}$] (bn2) at (3,7.25) {}; 
    \node[pn, label=above:$b_{n+3}$] (bn3) at (3,8.25) {};
    \node[pn, label=left:$a_{n+1}$] (an1) at (0,5.75) {}; 
    \node[pn, label=left:$a_{n+2}$] (an2) at (0,6.75) {};
    \node[pn, label=right:$c$] (c) at (5,1) {};
    \node[pn, label=right:$d$] (d) at (5,7.25) {};
    
    \draw (a1)--(a2) (a2)--(a3) (a3)--(a4);
    \draw (b1)--(b2) (b2)--(b3);
    \draw (an1)--(an2);
    \draw (bn)--(bn1) (bn1)--(bn2) (bn2)--(bn3);
    \draw (a1)--(b3);
    \draw (b1)--(a2) (b2)--(a3) (b3)--(a4);
    \draw (an1)--(bn3); 
    \draw (bn)--(an1) (bn1)--(an2); 
    \draw (b1)--(c);
    \draw (d)--(bn3);
    
    \draw (a2) -- (1.5, 2);
    \draw[dashed] (1.5, 2) -- (2.25, 2.5);
    \draw (a3) -- (1.5, 3);
    \draw[dashed] (1.5, 3) -- (2.25, 3.5);
    \draw (a4) -- (1.5, 4);
    \draw[dashed] (1.5, 4) -- (2.25, 4.5);
    
    \draw (bn2) -- (1.5, 6);
    \draw[dashed] (1.5, 6) -- (0.75, 5.375);
    \draw (bn1) -- (1.5, 5);
    \draw[dashed] (1.5, 5) -- (0.75, 4.375);
    
    \draw[dashed] (a4) -- (0, 3.8);
    \draw[dashed] (b3) -- (3, 2.8);
    \draw[dashed] (an1) -- (0, 4.95);
    \draw[dashed] (bn) -- (3, 4.45);
    
\end{tikzpicture}
}
\caption{The poset \(H_n\).}
\label{fig:kelly-Hn}
\end{figure}
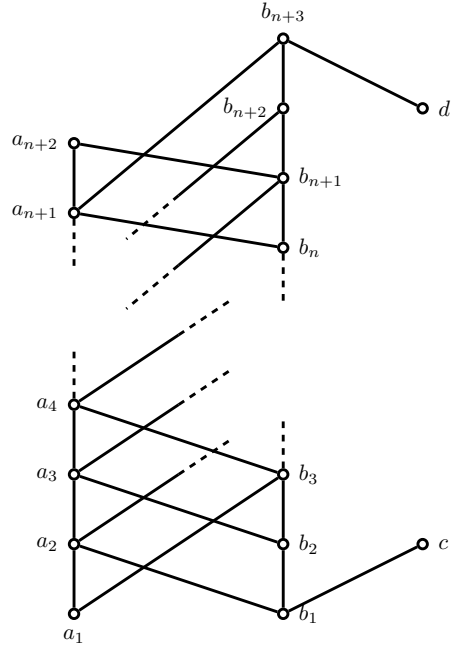

The following two tables define \(W_{H_n}\).
\par\nopagebreak[4]
\noindent
{\small
\renewcommand{\arraystretch}{1.08}
\begin{minipage}[t]{0.49\textwidth}
\vspace{0pt}
\centering
\begin{tabular}{@{}clc@{}}
\toprule
Vertex & Critical pair & Range\\
\midrule
\(\alpha_i\) & \(\bigl((0,a_i),(q,b_{i+1})\bigr)\) &
  \(1\leq i\leq n+2\)\\
\(\beta_i\) & \(\bigl((0,b_i),(q,a_i)\bigr)\) &
  \(1\leq i\leq n+2\)\\
\(\pi_1\) & \(\bigl((p,b_1),(q,b_{n+3})\bigr)\) & \\
\(\pi_2\) & \(\bigl((q,a_1),(p,b_{n+3})\bigr)\) & \\
\(\pi_3\) & \(\bigl((q,b_1),(p,a_{n+2})\bigr)\) & \\
\(\pi_4\) & \(\bigl((q,b_1),(p,b_{n+3})\bigr)\) & \\
\(\ell_1\) & \(\bigl((0,d),(p,a_{n+2})\bigr)\) & \\
\(\ell_2\) & \(\bigl((0,d),(q,a_{n+2})\bigr)\) & \\
\(\ell_3\) & \(\bigl((0,d),(p,c)\bigr)\) & \\
\(\ell_4\) & \(\bigl((0,d),(q,c)\bigr)\) & \\
\(\ell_5\) & \(\bigl((0,a_1),(p,c)\bigr)\) & \\
\(\ell_6\) & \(\bigl((0,a_1),(q,c)\bigr)\) & \\
\(r_1\) & \(\bigl((0,b_1),(q,d)\bigr)\) & \\
\(r_2\) & \(\bigl((0,c),(q,b_{n+3})\bigr)\) & \\
\bottomrule
\end{tabular}
\end{minipage}\hfill
\begin{minipage}[t]{0.49\textwidth}
\vspace{0pt}
\centering
\begin{tabular}{@{}lll@{}}
\toprule
Vertices & Adjacent vertices & Range\\
\midrule
\(\alpha_1\) & \(\beta_1,\beta_2,\pi_3,\pi_4\) & \\
\(\alpha_i\) & \(\beta_i,\beta_{i+1},\pi_2,\pi_3,\pi_4\) &
  \(2\leq i\leq n+1\)\\
\(\alpha_{n+2}\) & \(\beta_{n+2},\pi_3,\ell_1,\ell_2\) & \\
\(\beta_1\) & \(\pi_2,\ell_5,\ell_6\) & \\
\(\beta_i\) & \(\pi_2,\pi_3,\pi_4\) & \(2\leq i\leq n+1\)\\
\(\beta_{n+2}\) & \(\pi_2,\pi_4\) & \\
\(\pi_1\) & \(\pi_2,\pi_3,\pi_4,\ell_1,\ell_3,\ell_5\) & \\
\(\pi_2\) & \(\ell_2\) & \\
\(\pi_3\) & \(\ell_6\) & \\
\(\pi_4\) & \(\ell_2,\ell_4,\ell_6\) & \\
\(\ell_1,\ell_2\) & \(r_1\) & \\
\(\ell_3,\ell_4\) & \(r_1,r_2\) & \\
\(\ell_5,\ell_6\) & \(r_2\) & \\
\bottomrule
\end{tabular}
\end{minipage}
}

\begin{proposition}\label{prop:H-coloring}
For every \(n\geq0\),
\(W_{H_n}\subseteq\Gamma(\Vee\times H_n)\), and \(W_{H_n}\) is not
\(3\)-colorable.
\end{proposition}

\begin{proof}
\smallskip
\noindent\emph{Critical pairs.}
The minimal and maximal elements of \(H_n\) are
\[
 \Min(H_n)=\{a_1,b_1,d\},
 \qquad
 \Max(H_n)=\{a_{n+2},b_{n+3},c\}.
\]
In \(H_n\),
\[
 a_i\leq b_j\iff j\geq i+2,
 \qquad
 b_i\leq a_j\iff j\geq i+1.
\]
By these formulas, condition~\eqref{eq:product-critical-first} holds
for \(\alpha_i\) and \(\beta_i\):
\[
\begin{array}{llll}
 \alpha_i:& a_i\nleq b_{i+1},&
 \strictdown a_i\subseteq\downarrow b_{i+1},&
 \strictup b_{i+1}\subseteq\uparrow a_i,\\
 \beta_i:& b_i\nleq a_i,&
 \strictdown b_i\subseteq\downarrow a_i,&
 \strictup a_i\subseteq\uparrow b_i.
\end{array}
\]
For each of \(\ell_1,\ldots,\ell_6\), the first \(H_n\)-coordinate is
minimal, the second is maximal, and the two are incomparable.  Thus
condition~\eqref{eq:product-critical-first} holds for these six
vertices.  For \(r_1,r_2\), the same condition follows from
\[
\begin{array}{llll}
 r_1:& b_1\nleq d,&
 \strictdown b_1=\varnothing,&
 \strictup d=\{b_{n+3}\}\subseteq\uparrow b_1,\\
 r_2:& c\nleq b_{n+3},&
 \strictdown c=\{b_1\}\subseteq\downarrow b_{n+3},&
 \strictup b_{n+3}=\varnothing.
\end{array}
\]
By Lemma~\ref{lem:product-critical}\textup{(i)},
\(\alpha_i,\beta_i,\ell_1,\ldots,\ell_6,r_1,r_2\) are critical pairs.
Since \(a_1\leq b_{n+3}\), \(b_1\leq a_{n+2}\), and
\(b_1\leq b_{n+3}\),
Lemma~\ref{lem:product-critical}\textup{(ii)} shows that
\(\pi_1,\pi_2,\pi_3,\pi_4\) are critical pairs.

\smallskip
\noindent\emph{Edges.}
By the Hasse diagram of \(H_n\), the two comparisons in
\eqref{eq:product-edge} hold for every adjacency in the right table.
Hence
\[
 W_{H_n}\subseteq\Gamma(\Vee\times H_n).
\]

Suppose that \(\varphi\) is a proper \(3\)-coloring of \(W_{H_n}\) with
colors \(0,1,2\).  We may take \(\varphi(\pi_1)=0\).  Since
\(\pi_1\sim\pi_2,\pi_3,\pi_4\), we have
\[
 \varphi(\pi_2),\varphi(\pi_3),\varphi(\pi_4)\in\{1,2\}.
\]

\medskip
\noindent\textbf{Claim.}
The vertices \(\pi_2,\pi_3,\pi_4\) have the same color.

\smallskip
Assume that \(\pi_2,\pi_3,\pi_4\) do not all have the same color.  We may take
\(\varphi(\pi_2)=1\).  We consider the three possible pairs
\((\varphi(\pi_3),\varphi(\pi_4))\).

\smallskip
\noindent\emph{Case 1: \((1,2)\).}
\[
 \beta_{n+2}\sim\pi_2,\pi_4
 \quad\Longrightarrow\quad
 \varphi(\beta_{n+2})=0.
\]
If \(n\geq1\), the neighbors
\(\beta_{n+2},\pi_2,\pi_4\) of \(\alpha_{n+1}\) have colors
\(0,1,2\).  If \(n=0\), the neighbors
\(\beta_2,\pi_3,\pi_4\) of \(\alpha_1\) have colors \(0,1,2\).
Both alternatives are impossible.

\smallskip
\noindent\emph{Case 2: \((2,2)\).}
\[
\begin{aligned}
 \beta_{n+2}\sim\pi_2,\pi_4
   &\Longrightarrow \varphi(\beta_{n+2})=0,\\
 \ell_2\sim\pi_2,\pi_4
   &\Longrightarrow \varphi(\ell_2)=0,\\
 \alpha_{n+2}\sim\beta_{n+2},\pi_3
   &\Longrightarrow \varphi(\alpha_{n+2})=1,\\
 \ell_1\sim\alpha_{n+2},\pi_1
   &\Longrightarrow \varphi(\ell_1)=2,\\
 r_1\sim\ell_1,\ell_2
   &\Longrightarrow \varphi(r_1)=1,\\
 \ell_3\sim r_1,\pi_1
   &\Longrightarrow \varphi(\ell_3)=2,\\
 \ell_4\sim r_1,\pi_4
   &\Longrightarrow \varphi(\ell_4)=0,\\
 r_2\sim\ell_3,\ell_4
   &\Longrightarrow \varphi(r_2)=1.
\end{aligned}
\]
The edge \(\beta_1\sim\pi_2\) excludes color \(1\) from
\(\beta_1\).  If \(\varphi(\beta_1)=0\), the neighbors
\(\beta_1,r_2,\pi_3\) of \(\ell_6\) have colors \(0,1,2\).
If \(\varphi(\beta_1)=2\), the neighbors
\(\pi_1,r_2,\beta_1\) of \(\ell_5\) have colors \(0,1,2\).
Both choices are impossible.

\smallskip
\noindent\emph{Case 3: \((2,1)\).}
Suppose first that \(n\geq1\).  We have
\[
 \beta_{n+1}\sim\pi_2,\pi_3
 \quad\Longrightarrow\quad
 \varphi(\beta_{n+1})=0.
\]
The neighbors
\(\beta_{n+1},\pi_2,\pi_3\) of \(\alpha_{n+1}\) have colors
\(0,1,2\), a contradiction.

Suppose \(n=0\).  Then
\[
\begin{aligned}
 \alpha_1\sim\pi_3,\pi_4
   &\Longrightarrow \varphi(\alpha_1)=0,\\
 \beta_1\sim\alpha_1,\pi_2
   &\Longrightarrow \varphi(\beta_1)=2,\\
 \beta_2\sim\alpha_1,\pi_2
   &\Longrightarrow \varphi(\beta_2)=2.
\end{aligned}
\]
If \(\varphi(r_2)=1\), the neighbors \(\pi_1,r_2,\beta_1\) of
\(\ell_5\) have colors \(0,1,2\).  If \(\varphi(r_2)=0\), the neighbors
\(r_2,\pi_4,\beta_1\) of \(\ell_6\) have colors \(0,1,2\).  Thus
\(\varphi(r_2)=2\).  We have
\[
\begin{aligned}
 \ell_3\sim r_2,\pi_1
   &\Longrightarrow \varphi(\ell_3)=1,\\
 \ell_4\sim r_2,\pi_4
   &\Longrightarrow \varphi(\ell_4)=0,\\
 r_1\sim\ell_3,\ell_4
   &\Longrightarrow \varphi(r_1)=2,\\
 \ell_1\sim r_1,\pi_1
   &\Longrightarrow \varphi(\ell_1)=1,\\
 \ell_2\sim r_1,\pi_2
   &\Longrightarrow \varphi(\ell_2)=0.
\end{aligned}
\]
The neighbors \(\ell_2,\ell_1,\beta_2\) of \(\alpha_2\) have
colors \(0,1,2\), a contradiction.

All three cases are impossible. This proves the claim.
We may take
\[
 \varphi(\pi_2)=\varphi(\pi_3)=\varphi(\pi_4)=1.
\]
Every \(\beta_i\) is adjacent to \(\pi_2\), which has color \(1\).
Hence \(\varphi(\beta_i)\in\{0,2\}\).  For
\(1\leq i\leq n+1\), the vertex \(\alpha_i\) is adjacent to
\(\beta_i,\beta_{i+1},\pi_3\).  If \(\beta_i\) and
\(\beta_{i+1}\) had different colors, no color would be available for
\(\alpha_i\).  Thus \(\varphi(\beta_i)=\varphi(\beta_{i+1})\) for
every \(1\leq i\leq n+1\), and
\[
  \varphi(\beta_1)=\cdots=\varphi(\beta_{n+2})=b,
  \qquad b\in\{0,2\}.
\]
Both \(\beta_{n+2}\) and \(\ell_2\) are adjacent to \(\pi_2\), so
their colors belong to \(\{0,2\}\).  If their colors were different,
no color would be available for \(\alpha_{n+2}\), which is adjacent to
\(\beta_{n+2},\ell_2,\pi_3\).  Hence
\(\varphi(\ell_2)=\varphi(\beta_{n+2})=b\).

Suppose first that \(b=2\).  If \(\varphi(r_2)=0\), the neighbors
\(r_2,\pi_3,\beta_1\) of \(\ell_6\) have colors \(0,1,2\).  If
\(\varphi(r_2)=1\), the neighbors \(\pi_1,r_2,\beta_1\) of
\(\ell_5\) have colors \(0,1,2\).  Thus \(\varphi(r_2)=2\), and
\[
\begin{aligned}
 \ell_3\sim r_2,\pi_1
   &\Longrightarrow \varphi(\ell_3)=1,\\
 \ell_4\sim r_2,\pi_4
   &\Longrightarrow \varphi(\ell_4)=0.
\end{aligned}
\]
The neighbors \(\ell_4,\ell_3,\ell_2\) of \(r_1\) have colors
\(0,1,2\), a contradiction.

Suppose \(b=0\).  We have
\[
\begin{aligned}
 \alpha_{n+2}\sim\beta_{n+2},\ell_2,\pi_3
   &\Longrightarrow \varphi(\alpha_{n+2})=2,\\
 \ell_1\sim\alpha_{n+2},\pi_1
   &\Longrightarrow \varphi(\ell_1)=1.
\end{aligned}
\]
If \(\varphi(r_2)=2\), the neighbors \(\beta_1,\pi_3,r_2\) of
\(\ell_6\) have colors \(0,1,2\).  Hence
\(\varphi(r_2)\in\{0,1\}\).

If \(\varphi(r_2)=0\), then
\[
 \ell_4\sim r_2,\pi_4
 \quad\Longrightarrow\quad
 \varphi(\ell_4)=2.
\]
The neighbors \(\ell_2,\ell_1,\ell_4\) of \(r_1\) have colors
\(0,1,2\), a contradiction.  If \(\varphi(r_2)=1\), then
\[
 \ell_3\sim r_2,\pi_1
 \quad\Longrightarrow\quad
 \varphi(\ell_3)=2.
\]
The neighbors \(\ell_2,\ell_1,\ell_3\) of \(r_1\) again have colors
\(0,1,2\), a contradiction.

Thus \(W_{H_n}\) is not \(3\)-colorable.
\end{proof}

\subsection{The family \texorpdfstring{\(I_n\)}{I-n}}

\begin{figure}[htbp]
\centering
\resizebox{0.76\textwidth}{!}{%
\begin{tikzpicture}[scale=1.1]
    \node[pn, label=below:$a_1$] (a1) at (0,0) {}; 
    \node[pn, label=below:$a_2$] (a2) at (2,0) {}; 
    \node[pn, label=below:$a_3$] (a3) at (4,0) {}; 
    \node[pn, label=below:$a_{n+2}$] (an2) at (7,0) {};
    \node[pn, label=left:$b_1$] (b1) at (-1,1.5) {}; 
    \node[pn, label=left:$b_2$] (b2) at (1,1.5) {}; 
    \node[pn, label=right:$b_3$] (b3) at (3,1.5) {}; 
    \node[pn, label=right:$b_{n+2}$] (bn2) at (6,1.5) {};
    \node[pn, label=right:$b_{n+3}$] (bn3) at (8,1.5) {};
    \node[pn, label=above:$d_1$] (d1) at (2,3.5) {}; 
    \node[pn, label=above:$d_2$] (d2) at (5,3.5) {};
    \node[pn, label=above:$c$] (c) at (3.5,3) {};
    
    \draw (a1)--(b1) (a1)--(b2) (a2)--(b2) (a2)--(b3) (a3)--(b3);
    \draw (an2)--(bn2) (an2)--(bn3);
    \draw (b1)--(d1) (b2)--(d1) (b3)--(d1) (bn2)--(d1) (c)--(d1);
    \draw (b2)--(d2) (b3)--(d2) (bn2)--(d2) (bn3)--(d2) (c)--(d2);
    
    \draw (a3) -- (4.5,0.75);
    \draw (5.5,0.75) -- (bn2);
    \node at (5,0.75) {\Large\(\cdots\)};
\end{tikzpicture}
}
\caption{The poset \(I_n\).}
\label{fig:kelly-In}
\end{figure}
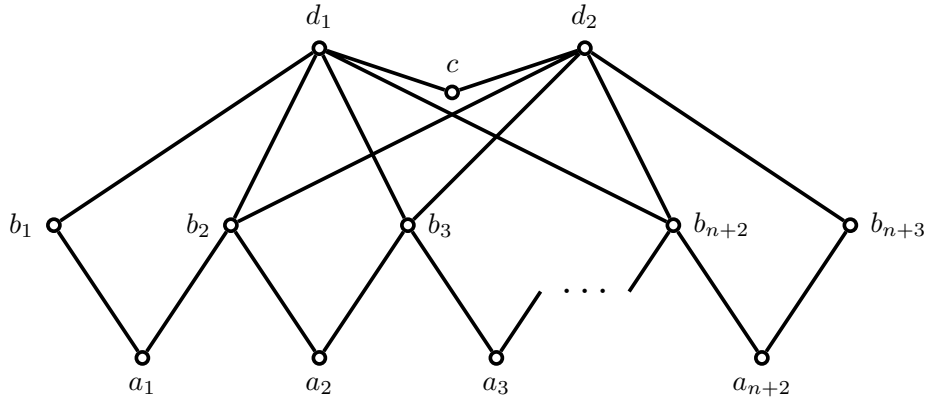
\FloatBarrier

The following two tables define \(W_{I_n}\).
\par\nopagebreak[4]
\noindent
{\small
\renewcommand{\arraystretch}{1.08}
\begin{minipage}[t]{0.49\textwidth}
\vspace{0pt}
\centering
\begin{tabular}{@{}clc@{}}
\toprule
Vertex & Critical pair & Range\\
\midrule
\(\alpha_i\) & \(\bigl((0,a_i),(q,c)\bigr)\) & \(1\leq i\leq n+2\)\\
\(\beta_i\) & \(\bigl((0,c),(q,b_i)\bigr)\) & \(1\leq i\leq n+3\)\\
\(\gamma_i\) & \(\bigl((q,a_i),(p,d_2)\bigr)\) & \(1\leq i\leq n+1\)\\
\(\gamma_{n+2}\) & \(\bigl((q,a_{n+2}),(p,d_1)\bigr)\) & \\
\(\pi_1\) & \(\bigl((q,c),(p,d_1)\bigr)\) & \\
\(\pi_2\) & \(\bigl((q,c),(p,d_2)\bigr)\) & \\
\(\pi_3\) & \(\bigl((p,c),(q,d_2)\bigr)\) & \\
\(r_1\) & \(\bigl((0,b_1),(q,d_2)\bigr)\) & \\
\(r_2\) & \(\bigl((0,b_{n+3}),(q,d_1)\bigr)\) & \\
\bottomrule
\end{tabular}
\end{minipage}\hfill
\begin{minipage}[t]{0.49\textwidth}
\vspace{0pt}
\centering
\begin{tabular}{@{}lll@{}}
\toprule
Vertices & Adjacent vertices & Range\\
\midrule
\(\alpha_i\) & \(\beta_i,\beta_{i+1},\pi_1,\pi_2\) &
  \(1\leq i\leq n+2\)\\
\(\beta_1\) & \(\gamma_1,r_1\) & \\
\(\beta_i\) & \(\gamma_{i-1},\gamma_i\) & \(2\leq i\leq n+1\)\\
\(\beta_{n+2}\) & \(\gamma_{n+1},\gamma_{n+2}\) & \\
\(\beta_{n+3}\) & \(\gamma_{n+2},r_2\) & \\
\(\gamma_i\) & \(\pi_3,r_2\) & \(1\leq i\leq n+1\)\\
\(\gamma_{n+2},\pi_1\) & \(\pi_3,r_1\) & \\
\(\pi_2\) & \(\pi_3,r_2\) & \\
\(r_1\) & \(r_2\) & \\
\bottomrule
\end{tabular}
\end{minipage}
}

\begin{proposition}\label{prop:I-coloring}
For every \(n\geq0\),
\(W_{I_n}\subseteq\Gamma(\Vee\times I_n)\), and \(W_{I_n}\) is not
\(3\)-colorable.
\end{proposition}

\begin{proof}
\enlargethispage{\baselineskip}
\smallskip
\noindent\emph{Critical pairs.}
The minimal and maximal elements of \(I_n\) are
\[
 \Min(I_n)=\{a_i:1\leq i\leq n+2\}\cup\{c\},
 \qquad
 \Max(I_n)=\{d_1,d_2\}.
\]
For \(\alpha_i\), condition~\eqref{eq:product-critical-first} follows
from
\[
 a_i\nleq c,\qquad \strictdown a_i=\varnothing,
 \qquad \strictup c=\{d_1,d_2\}\subseteq\uparrow a_i.
\]
For \(\beta_i\), the same condition follows from
\[
 c\nleq b_i,\qquad \strictdown c=\varnothing,
 \qquad \strictup b_i\subseteq\{d_1,d_2\}\subseteq\uparrow c.
\]
For \(r_1,r_2\), condition~\eqref{eq:product-critical-first} follows
from
\[
\begin{array}{llll}
 r_1:& b_1\nleq d_2,&
 \strictdown b_1=\{a_1\}\subseteq\downarrow d_2,&
 \strictup d_2=\varnothing,\\
 r_2:& b_{n+3}\nleq d_1,&
 \strictdown b_{n+3}=\{a_{n+2}\}\subseteq\downarrow d_1,&
 \strictup d_1=\varnothing.
\end{array}
\]
By Lemma~\ref{lem:product-critical}\textup{(i)},
\(\alpha_i,\beta_i,r_1,r_2\) are critical pairs.
Since \(a_i\leq d_2\) for \(1\leq i\leq n+1\),
\(a_{n+2}\leq d_1\), \(c\leq d_1\), and \(c\leq d_2\),
Lemma~\ref{lem:product-critical}\textup{(ii)} shows that
\(\gamma_1,\ldots,\gamma_{n+2},\pi_1,\pi_2,\pi_3\) are critical pairs.

\smallskip
\noindent\emph{Edges.}
By the Hasse diagram of \(I_n\), the two \(I_n\)-comparisons in
\eqref{eq:product-edge} hold for every adjacency in the right table.
Hence
\[
 W_{I_n}\subseteq\Gamma(\Vee\times I_n).
\]

Suppose that \(\varphi\) is a proper \(3\)-coloring of \(W_{I_n}\)
with colors \(0,1,2\).  Since \(\pi_1\sim\pi_3\), we may assume
\[
 \varphi(\pi_3)=0,
 \qquad
 \varphi(\pi_1)=1.
\]
Each of \(\gamma_1,\ldots,\gamma_{n+2},\pi_2\) is adjacent to
\(\pi_3\), so these vertices have color \(1\) or \(2\).

\medskip
\noindent\textbf{Claim.}
\(\varphi(\pi_2)=1\).

\smallskip
Assume that \(\varphi(\pi_2)=2\).  For \(1\leq i\leq n+2\),
\[
 \alpha_i\sim\pi_1,\pi_2
 \quad\Longrightarrow\quad
 \varphi(\alpha_i)=0.
\]
The adjacencies \(\beta_1\sim\alpha_1\) and
\(\beta_i\sim\alpha_{i-1}\) for \(2\leq i\leq n+3\) exclude color
\(0\) from every \(\beta_i\).

The edges
\[
 \beta_i\sim\gamma_i
 \quad\text{and}\quad
 \gamma_i\sim\beta_{i+1}
 \qquad(1\leq i\leq n+2)
\]
form a path from \(\beta_1\) to \(\beta_{n+3}\).  Every vertex on the
path has color \(1\) or \(2\), so the colors alternate.  The subpaths from
\(\gamma_1\) to \(\beta_{n+3}\) and from \(\beta_1\) to
\(\gamma_{n+2}\) each have \(2n+3\) edges.  Thus
\[
  \{\varphi(\beta_{n+3}),\varphi(\gamma_1)\}=\{1,2\},
  \qquad
  \{\varphi(\beta_1),\varphi(\gamma_{n+2})\}=\{1,2\}.
\]
The two neighbors \(\beta_{n+3},\gamma_1\) of \(r_2\) use colors
\(1,2\), so \(\varphi(r_2)=0\).  The neighbors
\(r_2,\beta_1,\gamma_{n+2}\) of \(r_1\) use all three colors, a
contradiction.  This proves the claim.

Since \(r_2\sim\pi_2\), we have
\(\varphi(r_2)\in\{0,2\}\).

Suppose first that \(\varphi(r_2)=2\).  We have
\[
\begin{aligned}
 r_1\sim r_2,\pi_1
   &\Longrightarrow \varphi(r_1)=0,\\
 \gamma_i\sim r_2,\pi_3
   &\Longrightarrow \varphi(\gamma_i)=1
     &&(1\leq i\leq n+1),\\
 \beta_1\sim r_1,\gamma_1
   &\Longrightarrow \varphi(\beta_1)=2.
\end{aligned}
\]
For \(1\leq i\leq n+1\), if \(\varphi(\beta_i)=2\), then
\[
\begin{aligned}
 \alpha_i\sim\beta_i,\pi_1
   &\Longrightarrow \varphi(\alpha_i)=0,\\
 \beta_{i+1}\sim\alpha_i,\gamma_i
   &\Longrightarrow \varphi(\beta_{i+1})=2.
\end{aligned}
\]
Starting with \(\beta_1\), forward induction gives
\[
 \varphi(\beta_1)=\cdots=\varphi(\beta_{n+2})=2.
\]
At the right endpoint,
\[
\begin{aligned}
 \gamma_{n+2}\sim\beta_{n+2},\pi_3
   &\Longrightarrow \varphi(\gamma_{n+2})=1,\\
 \beta_{n+3}\sim\gamma_{n+2},r_2
   &\Longrightarrow \varphi(\beta_{n+3})=0.
\end{aligned}
\]
The neighbors \(\beta_{n+3},\pi_1,\beta_{n+2}\) of
\(\alpha_{n+2}\) use colors \(0,1,2\), a contradiction.

Suppose that \(\varphi(r_2)=0\).  We have
\[
\begin{aligned}
 r_1\sim r_2,\pi_1
   &\Longrightarrow \varphi(r_1)=2,\\
 \gamma_{n+2}\sim r_1,\pi_3
   &\Longrightarrow \varphi(\gamma_{n+2})=1,\\
 \beta_{n+3}\sim\gamma_{n+2},r_2
   &\Longrightarrow \varphi(\beta_{n+3})=2.
\end{aligned}
\]
The edge \(\beta_{n+2}\sim\gamma_{n+2}\) excludes color \(1\) from
\(\beta_{n+2}\).  If \(\varphi(\beta_{n+2})=0\), the neighbors
\(\beta_{n+2},\pi_1,\beta_{n+3}\) of \(\alpha_{n+2}\) have colors
\(0,1,2\), a contradiction.  Thus
\(\varphi(\beta_{n+2})=2\).

For \(1\leq i\leq n+1\), if
\(\varphi(\beta_{i+1})=2\), then
\[
\begin{aligned}
 \gamma_i\sim\beta_{i+1},\pi_3
   &\Longrightarrow \varphi(\gamma_i)=1,\\
 \alpha_i\sim\beta_{i+1},\pi_1
   &\Longrightarrow \varphi(\alpha_i)=0,\\
 \beta_i\sim\gamma_i,\alpha_i
   &\Longrightarrow \varphi(\beta_i)=2.
\end{aligned}
\]
Starting with \(\beta_{n+2}\) and taking
\(i=n+1,n,\ldots,1\), reverse induction gives
\[
 \varphi(\beta_{n+2})=\cdots=\varphi(\beta_1)=2.
\]
This contradicts the edge \(r_1\sim\beta_1\), since
\(\varphi(r_1)=2\).
\end{proof}

\begin{proof}[Proof of Theorem~\ref{thm:infinite}]
Propositions~\ref{prop:E-coloring}, \ref{prop:F-coloring},
\ref{prop:GJ-coloring}, \ref{prop:H-coloring}, and
\ref{prop:I-coloring} provide a non-\(3\)-colorable subgraph
\(W_R\subseteq\Gamma(\Vee\times R)\) for every member \(R\) of the six
infinite families and every \(n\geq0\).  Therefore
\(\Gamma(\Vee\times R)\) is not \(3\)-colorable.
\end{proof}

\section{Computer verification for the ten fixed posets}
\label{sec:finite-cases}

The ten fixed noncrown posets in the classification are
\[
  B,C,D,CX_1,CX_2,CX_3,EX_1,EX_2,FX_1,FX_2.
\]
Figure~\ref{fig:kelly-fixed} shows the ten Hasse diagrams.
Table~\ref{tab:fixed-relations} lists the generating relations supplied
to the program.
\begin{figure}[p]
\centering
\resizebox{0.96\textwidth}{!}{%
\begin{tikzpicture}[scale=1.0]
    \begin{scope}[xshift=0cm]
        \node[pn, label=below:$b_1$] (Bb1) at (0.5,0) {};
        \node[pn, label=below:$a$] (Ba) at (2,0) {};
        \node[pn, label=below:$b_2$] (Bb2) at (3,0) {};
        \node[pn, label=below:$b_3$] (Bb3) at (4.5,0) {};
        \node[pn, label=above:$c_1$] (Bc1) at (1,2) {};
        \node[pn, label=above:$c_2$] (Bc2) at (2,2) {};
        \node[pn, label=above:$c_3$] (Bc3) at (3,2) {};
        \draw (Bb1)--(Bc1) (Ba)--(Bc1) (Ba)--(Bc2)
              (Ba)--(Bc3) (Bb2)--(Bc2) (Bb3)--(Bc3);
        \node[title] at (2.5,-0.9) {$B$};
    \end{scope}

    \begin{scope}[xshift=6cm]
        \node[pn, label=below:$a$] (Ca) at (2,0) {};
        \node[pn, label=left:$b_1$] (Cb1) at (1,1) {};
        \node[pn, label=right:$b_2$] (Cb2) at (3,1) {};
        \node[pn, label=right:$b_3$] (Cb3) at (4.5,1) {};
        \node[pn, label=above:$c_1$] (Cc1) at (0,2) {};
        \node[pn, label=above:$c_2$] (Cc2) at (2,2) {};
        \node[pn, label=above:$c_3$] (Cc3) at (4,2) {};
        \draw (Ca)--(Cb1) (Ca)--(Cb2) (Cb1)--(Cc1)
              (Cb1)--(Cc2) (Cb2)--(Cc2) (Cb2)--(Cc3)
              (Cb3)--(Cc2);
        \node[title] at (2,-0.9) {$C$};
    \end{scope}

    \begin{scope}[xshift=12cm]
        \node[pn, label=below:$a$] (Da) at (2,0) {};
        \node[pn, label=left:$b_1$] (Db1) at (0,1.5) {};
        \node[pn, label=above right:$b_2$] (Db2) at (2,1.8) {};
        \node[pn, label=right:$b_3$] (Db3) at (4,1.5) {};
        \node[pn, label=above:$c_1$] (Dc1) at (0,3) {};
        \node[pn, label=above:$c_2$] (Dc2) at (4,3) {};
        \draw (Da)--(Db1) (Da)--(Db3) (Db1)--(Dc1)
              (Db2)--(Dc1) (Db2)--(Dc2) (Db3)--(Dc2);
        \node[title] at (2,-0.9) {$D$};
    \end{scope}
\end{tikzpicture}
}
\par\medskip
\resizebox{0.94\textwidth}{!}{%
\begin{tikzpicture}[scale=1.0]
    \begin{scope}[xshift=0cm]
        \node[pn, label=below:$a_1$] (XOnea1) at (0,0) {};
        \node[pn, label=below:$a_2$] (XOnea2) at (3,0) {};
        \node[pn, label=below:$a_3$] (XOnea3) at (5,0) {};
        \node[pn, label=above:$b_1$] (XOneb1) at (0,1.5) {};
        \node[pn, label=right:$b_2$] (XOneb2) at (3,1.5) {};
        \node[pn, label=above:$b_3$] (XOneb3) at (4,1.5) {};
        \node[pn, label=right:$c$] (XOnec) at (3,3) {};
        \draw (XOnea1)--(XOneb1) (XOnea1)--(XOnec)
              (XOnea2)--(XOneb1) (XOnea2)--(XOneb2)
              (XOnea2)--(XOneb3) (XOnea3)--(XOneb2)
              (XOneb2)--(XOnec);
        \node[title] at (2.5,-0.9) {$CX_1$};
    \end{scope}

    \begin{scope}[xshift=6.5cm]
        \node[pn, label=below:$a_1$] (XTwoa1) at (0,0) {};
        \node[pn, label=below:$a_2$] (XTwoa2) at (2,0) {};
        \node[pn, label=below:$a_3$] (XTwoa3) at (4,0) {};
        \node[pn, label=above:$b_1$] (XTwob1) at (0,1.5) {};
        \node[pn, label=above:$b_2$] (XTwob2) at (1.5,1.5) {};
        \node[pn, label=above:$b_3$] (XTwob3) at (4,1.5) {};
        \node[pn, label=above:$c$] (XTwoc) at (2.5,1.5) {};
        \draw (XTwoa1)--(XTwob1) (XTwoa1)--(XTwob2)
              (XTwoa2)--(XTwob1) (XTwoa2)--(XTwob2)
              (XTwoa2)--(XTwob3) (XTwoa2)--(XTwoc)
              (XTwoa3)--(XTwob2) (XTwoa3)--(XTwob3);
        \node[title] at (2,-0.9) {$CX_2$};
    \end{scope}

    \begin{scope}[xshift=12cm]
        \node[pn, label=below:$a_1$] (XThreea1) at (0,0) {};
        \node[pn, label=below:$a_2$] (XThreea2) at (2,0) {};
        \node[pn, label=below:$a_3$] (XThreea3) at (4,0) {};
        \node[pn, label=above:$b_1$] (XThreeb1) at (0,1) {};
        \node[pn, label=above:$b_2$] (XThreeb2) at (2,1) {};
        \node[pn, label=above:$b_3$] (XThreeb3) at (3,1) {};
        \node[pn, label=right:$c$] (XThreec) at (4,2) {};
        \draw (XThreea1)--(XThreeb1) (XThreea1)--(XThreeb2)
              (XThreea2)--(XThreeb1) (XThreea2)--(XThreeb2)
              (XThreea2)--(XThreeb3) (XThreea3)--(XThreeb2)
              (XThreea3)--(XThreec) (XThreeb3)--(XThreec);
        \node[title] at (2,-0.9) {$CX_3$};
    \end{scope}
\end{tikzpicture}
}
\par\medskip
\resizebox{0.86\textwidth}{!}{%
\begin{tikzpicture}[scale=1.1]
    \begin{scope}[yshift=0cm, xshift=0cm]
        \node[pn, label=below:$a_1$] (a1) at (1,0) {}; 
        \node[pn, label=below:$a_2$] (a2) at (3,0) {}; 
        \node[pn, label=below:$a_3$] (a3) at (5,0) {};
        \node[pn, label=above:$b_1$] (b1) at (0,1.5) {}; 
        \node[pn, label=above:$b_2$] (b2) at (2,1.5) {}; 
        \node[pn, label=above:$b_3$] (b3) at (4,1.5) {}; 
        \node[pn, label=above:$b_4$] (b4) at (6,1.5) {};
        
        \draw (a1)--(b1) (a1)--(b2) (a1)--(b3) (a2)--(b2) (a2)--(b3) (a2)--(b4) (a3)--(b3);
        \node[title] at (3,-1) {$EX_1$};
    \end{scope}

    \begin{scope}[yshift=0cm, xshift=8cm]
        \node[pn, label=below:$a_1$] (a1) at (0,0) {}; 
        \node[pn, label=below:$a_2$] (a2) at (2,0) {}; 
        \node[pn, label=below:$a_3$] (a3) at (4,0) {};
        \node[pn, label=left:$b_1$] (b1) at (0,1.5) {}; 
        \node[pn, label=above:$b_2$] (b2) at (2,1.5) {}; 
        \node[pn, label=above:$b_3$] (b3) at (4,1.5) {};
        \node[pn, label=above:$c$] (c) at (0,3) {};
        
        \draw (a1)--(b1) (a1)--(b2) (a2)--(c) (a2)--(b2) (a2)--(b3) (a3)--(b2) (b1)--(c);
        \node[title] at (2,-1) {$EX_2$};
    \end{scope}
\end{tikzpicture}
}
\par\medskip
\resizebox{0.82\textwidth}{!}{%
\begin{tikzpicture}[scale=1.1]
    \begin{scope}[yshift=0cm, xshift=0cm]
        \node[pn, label=below:$a_1$] (a1) at (0,0) {}; 
        \node[pn, label=below:$a_2$] (a2) at (2,0) {}; 
        \node[pn, label=below:$a_3$] (a3) at (4,0) {};
        \node[pn, label=above:$b_1$] (b1) at (0,1.5) {}; 
        \node[pn, label=left:$b_2$] (b2) at (2,1.5) {}; 
        \node[pn, label=above:$b_3$] (b3) at (4,1.5) {};
        \node[pn, label=above:$c$] (c) at (2,3) {};
        
        \draw (a1)--(b1) (a1)--(b2) (a2)--(b1) (a2)--(b2) (a2)--(b3) (a3)--(b2) (b2)--(c) (b3)--(c);
        \node[title] at (2,-1) {$FX_1$};
    \end{scope}

    \begin{scope}[yshift=0cm, xshift=7cm]
        \node[pn, label=below:$a_1$] (a1) at (0,0) {}; 
        \node[pn, label=below:$a_2$] (a2) at (2,0) {}; 
        \node[pn, label=below:$a_3$] (a3) at (4,0) {};
        \node[pn, label=above:$b_1$] (b1) at (0,1.5) {}; 
        \node[pn, label=left:$b_2$] (b2) at (2,1.5) {}; 
        \node[pn, label=above:$b_3$] (b3) at (4,1.5) {};
        \node[pn, label=above:$c$] (c) at (2,3) {};
        
        \draw (a1)--(b1) (a1)--(b3) (a1)--(c) (a2)--(b1) (a2)--(b2) (a2)--(b3) (a3)--(c) (a3)--(b3) (b2)--(c);
        \node[title] at (2,-1) {$FX_2$};
    \end{scope}
\end{tikzpicture}
}
\caption{The ten fixed posets in Kelly's classification.}
\label{fig:kelly-fixed}
\end{figure}
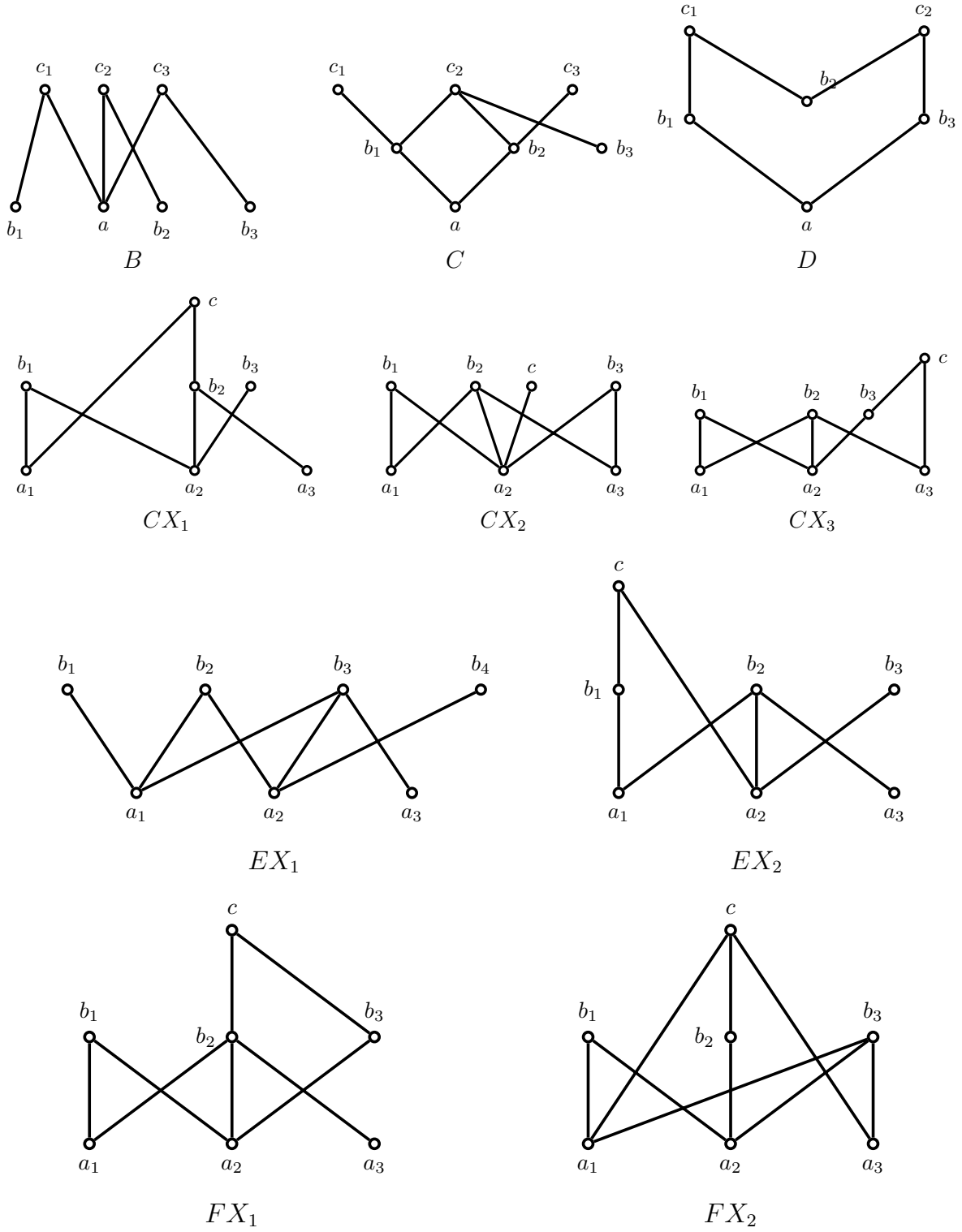

\begin{longtable}{@{}>{\(}l<{\)}p{0.81\linewidth}@{}}
\caption{Generating relations for the ten fixed
posets.}\label{tab:fixed-relations}\\
\toprule
\text{Poset}&\text{Generating relations}\\
\midrule
\endfirsthead
\toprule
\text{Poset}&\text{Generating relations}\\
\midrule
\endhead
B&\(b_1<c_1;\ a<c_1,c_2,c_3;\ b_2<c_2;\ b_3<c_3.\)\\
C&\(a<b_1,b_2;\ b_1<c_1,c_2;\ b_2<c_2,c_3;\ b_3<c_2.\)\\
D&\(a<b_1,b_3;\ b_1<c_1;\ b_2<c_1,c_2;\ b_3<c_2.\)\\
CX_1&\(a_1<b_1,c;\ a_2<b_1,b_2,b_3;\ a_3<b_2;\ b_2<c.\)\\
CX_2&\(a_1<b_1,b_2;\ a_2<b_1,b_2,b_3,c;\ a_3<b_2,b_3.\)\\
CX_3&\(a_1<b_1,b_2;\ a_2<b_1,b_2,b_3;\ a_3<b_2,c;\ b_3<c.\)\\
EX_1&\(a_1<b_1,b_2,b_3;\ a_2<b_2,b_3,b_4;\ a_3<b_3.\)\\
EX_2&\(a_1<b_1,b_2;\ a_2<c,b_2,b_3;\ a_3<b_2;\ b_1<c.\)\\
FX_1&\(a_1<b_1,b_2;\ a_2<b_1,b_2,b_3;\ a_3<b_2;\ b_2<c;\ b_3<c.\)\\
FX_2&\(a_1<b_1,b_3,c;\ a_2<b_1,b_2,b_3;\ a_3<c,b_3;\ b_2<c.\)\\
\bottomrule
\end{longtable}

The graph sizes and the outcomes of the exhaustive \(3\)-coloring
search are listed in Table~\ref{tab:fixed-graph-sizes}.

\begin{table}[!htbp]
\centering
\caption{Sizes of the graphs of critical pairs for the ten fixed posets.}
\label{tab:fixed-graph-sizes}
\small
\begin{tabular}{@{}lrrl@{}}
\toprule
Product&Vertices&Edges&Result\\
\midrule
\(\Vee\times B\)&30&72&not \(3\)-colorable\\
\(\Vee\times C\)&22&58&not \(3\)-colorable\\
\(\Vee\times D\)&18&56&not \(3\)-colorable\\
\(\Vee\times CX_1\)&30&92&not \(3\)-colorable\\
\(\Vee\times CX_2\)&38&128&not \(3\)-colorable\\
\(\Vee\times CX_3\)&30&101&not \(3\)-colorable\\
\(\Vee\times EX_1\)&38&111&not \(3\)-colorable\\
\(\Vee\times EX_2\)&30&88&not \(3\)-colorable\\
\(\Vee\times FX_1\)&28&89&not \(3\)-colorable\\
\(\Vee\times FX_2\)&38&156&not \(3\)-colorable\\
\bottomrule
\end{tabular}
\end{table}
\FloatBarrier

\begin{proposition}\label{prop:sporadic-computation}
For every
\[
 R\in\{B,C,D,CX_1,CX_2,CX_3,EX_1,EX_2,FX_1,FX_2\},
\]
the graph \(\Gamma(\Vee\times R)\) is not \(3\)-colorable.
\end{proposition}

\begin{proof}
Fix \(R\) and put \(P_R=\Vee\times R\).  The program performs the
following four steps.

\begin{enumerate}[label=\textup{(\arabic*)}]
\item It computes the reflexive transitive closure of the generating
relations in Table~\ref{tab:fixed-relations} and checks antisymmetry.
The resulting relation is the order relation \(\leq_R\) of \(R\).

\item It constructs \(P_R\) with the coordinatewise order
\[
  (s,x)\leq_{P_R}(t,y)
  \quad\Longleftrightarrow\quad
  s\leq_{\Vee}t\ \text{and}\ x\leq_R y.
\]
For each ordered pair \((u,v)\) of incomparable elements of \(P_R\), it
tests whether
\[
  \strictdown u\subseteq\strictdown v
  \qquad\text{and}\qquad
  \strictup v\subseteq\strictup u.
\]
By \eqref{eq:critical-pair}, the pairs satisfying these conditions are
exactly the elements of \(\Crit(P_R)\).

\item It constructs a graph \(G_R\) with vertex set \(\Crit(P_R)\).
Distinct vertices \((u_1,v_1)\) and \((u_2,v_2)\) are adjacent in \(G_R\)
exactly when
\[
  u_1\leq_{P_R}v_2
  \qquad\text{and}\qquad
  u_2\leq_{P_R}v_1.
\]
By \eqref{eq:critical-edge},
\[
  G_R=\Gamma(P_R).
\]

\item It tests \(G_R\) for \(3\)-colorability by exhaustive backtracking.
Whenever uncolored vertices remain, the program selects one of them and
tries every color in \(\{0,1,2\}\) that is not used by any of its colored
neighbors.  If one choice cannot be extended, the program tries the next
available color.  If no choice succeeds, it leaves the vertex uncolored
and returns to the preceding step.  The search succeeds when every vertex
has been colored.  The next vertex is selected by maximizing the number
of colors already used on its colored neighbors, with degree and the order
in the program's list used to break ties.  This rule affects only the order
in which the possibilities are examined.
\end{enumerate}

The search in \textup{(4)} is exhaustive.  Suppose that \(G_R\) has a
proper \(3\)-coloring \(\varphi\).  Starting at the root of the search
tree, follow the branch that assigns the color \(\varphi(v)\) whenever
the program selects a vertex \(v\).  Since \(\varphi\) is proper,
\(\varphi(v)\) is not used by any previously colored neighbor of \(v\).
Thus this color is among those tried by the program.  Continuing in this
way produces a complete coloring, so the search would succeed.

For each of the ten choices of \(R\), the search terminates without a
proper \(3\)-coloring.  The resulting numbers of vertices and edges are
listed in Table~\ref{tab:fixed-graph-sizes}.  Hence
\(\Gamma(\Vee\times R)\) is not \(3\)-colorable in every case.
\end{proof}

\section{Proofs of the main results}\label{sec:main-results}

\begin{theorem}\label{thm:irreducible-case}
Let \(R\) be a noncrown \(3\)-irreducible poset.  At least one of
\[
  \Gamma(\Vee\times R)
  \quad\text{and}\quad
  \Gamma(\Vee^d\times R)
\]
is not \(3\)-colorable.
\end{theorem}

\begin{proof}
By the classification, choose a poset \(S\) from the six infinite families
or the ten fixed posets such that \(R\cong S\) or \(R^d\cong S\).
Theorem~\ref{thm:infinite} covers the infinite families, and
Proposition~\ref{prop:sporadic-computation} covers the fixed posets.  Thus
\(\Gamma(\Vee\times S)\) is not \(3\)-colorable.

If \(R\cong S\), then \(\Gamma(\Vee\times R)\) is not
\(3\)-colorable.  If \(R^d\cong S\), then
\[
  (\Vee^d\times R)^d\cong\Vee\times S.
\]
Lemma~\ref{lem:graph-duality} shows that
\(\Gamma(\Vee^d\times R)\) is not \(3\)-colorable.
\end{proof}

\begin{corollary}\label{cor:crown-noncrown}
Let \(R\) be a \(3\)-irreducible poset that is not a crown.  For every
\(k\geq3\),
\[
  \dim(C_k\times R)=4.
\]
\end{corollary}

\begin{proof}
The result follows from Theorem~\ref{thm:irreducible-case} and
Proposition~\ref{prop:crown-product}.
\end{proof}

\begin{theorem}\label{thm:main}
Let \(P\) and \(Q\) be finite posets.  If \(\dim P=\dim Q=3\), then
\[
  \dim(P\times Q)\geq4.
\]
\end{theorem}

\begin{proof}
Choose \(3\)-irreducible subposets \(P_0\subseteq P\) and
\(Q_0\subseteq Q\).  By monotonicity, it suffices to prove
\(\dim(P_0\times Q_0)\geq4\).

If both \(P_0\) and \(Q_0\) are crowns, the result follows from
\cite[Proposition~19]{OrderGridsProducts}.  If exactly one is a crown,
Corollary~\ref{cor:crown-noncrown} applies.  Suppose neither is a crown.
Each contains a subposet isomorphic to one of \(Z_3,Z_3^d,C_2\).
Choose such subposets \(X\subseteq P_0\) and \(Y\subseteq Q_0\).
Up to factor exchange and duality, \(X\times Y\) is one of
\[
  Z_3\times Z_3,
  \qquad Z_3\times Z_3^d,
  \qquad Z_3\times C_2,
  \qquad C_2\times C_2.
\]
Each product has dimension at least four~\cite{OrderGridsProducts}.
Hence
\[
  4\leq\dim(X\times Y)
  \leq\dim(P_0\times Q_0),
\]
as required.
\end{proof}

\begin{corollary}\label{cor:crown-any}
Let \(P\) be a finite poset of dimension three.  For every \(k\geq3\),
\[
  \dim(C_k\times P)=4.
\]
\end{corollary}

\begin{proof}
Theorem~\ref{thm:main} and \eqref{eq:fmw-crown-upper-intro} give
\[
  4\leq\dim(C_k\times P)=\dim(P\times C_k)
  \leq\max\{3,\dim P+1\}=4.
\]
\end{proof}

\section*{Declarations}

\paragraph{Program availability.}
The supplementary archive contains
\nolinkurl{verify_fixed_cases_reference.py}, the Python program used in
Section~\ref{sec:finite-cases}, together with the commands and expected output
needed to reproduce Table~\ref{tab:fixed-graph-sizes}.  The input relations
are included in the program.  The program uses only the Python standard
library.  No external data are used.

\paragraph{Use of generative AI.}
During the preparation of this work, the authors used generative AI to improve
the readability and language of the manuscript and to assist with writing the
Python verification program for the ten fixed posets.  
The authors carefully reviewed and verified the manuscript and take full responsibility for its content, including the correctness of all mathematical statements, proofs, and references.


\begin{thebibliography}{99}

\bibitem{BarreraCruzEtAl2019}
F. Barrera-Cruz, R. Garcia, P. Harris, B. Kubik, H. Smith,
S. Talbott, L. Taylor, and W. T. Trotter,
The graph of critical pairs of a crown.
\textit{Order} 36 (2019), 621--652.

\bibitem{Bergman2026}
G.M. Bergman,
Some frustrating questions on dimensions of products of posets.
\textit{Discrete Math.} 349 (2026), Article 115002.

\bibitem{DushnikMiller1941}
B. Dushnik and E.W. Miller,
Partially ordered sets.
\textit{Amer. J. Math.} 63 (1941), 600--610.

\bibitem{OrderGridsProducts}
S. Felsner, T. M\"utze, and M. Wittmann,
Order dimension, grids, and products.
\textit{Order} 42 (2025), 811--827.

\bibitem{FelsnerTrotter2000}
S. Felsner and W. T. Trotter,
Dimension, graph and hypergraph coloring.
\textit{Order} 17 (2000), 167--177.

\bibitem{Kelly1977}
D. Kelly,
The \(3\)-irreducible partially ordered sets.
\textit{Canad. J. Math.} 29 (1977), 367--383.

\bibitem{KellyTrotter1982}
D. Kelly and W. T. Trotter,
Dimension theory for ordered sets.
In: I. Rival (Ed.), \textit{Ordered Sets},
D. Reidel, 1982, pp. 171--211.

\bibitem{Reuter1989}
K. Reuter,
On the dimension of the Cartesian product of relations and orders.
\textit{Order} 6 (1989), 277--293.

\bibitem{TrotterCrown1974}
W. T. Trotter,
Dimension of the crown \(S_n^k\).
\textit{Discrete Math.} 8 (1974), 85--103.

\bibitem{Trotter1985}
W. T. Trotter,
The dimension of the Cartesian product of partial orders.
\textit{Discrete Math.} 53 (1985), 255--263.

\bibitem{Trotter1992}
W. T. Trotter,
\textit{Combinatorics and Partially Ordered Sets: Dimension Theory}.
Johns Hopkins University Press, Baltimore, 1992.

\bibitem{TrotterMoore1976}
W. T. Trotter and J. I. Moore,
Characterization problems for graphs, partially ordered sets,
lattices, and families of sets.
\textit{Discrete Math.} 16 (1976), 361--381.

\end{thebibliography}
\end{document}